\documentclass[11pt]{amsart}
\usepackage{amssymb}
\usepackage{graphicx}
\usepackage{xcolor} 
\usepackage{tensor}

\usepackage[margin=1.25in]{geometry}

\definecolor{green}{rgb}{0,0.8,0} 

\newtheorem{theorem}{Theorem}[section]
\newtheorem{corollary}[theorem]{Corollary}
\newtheorem{lemma}[theorem]{Lemma}
\newtheorem{proposition}[theorem]{Proposition}
\theoremstyle{definition}

\theoremstyle{remark}
\newtheorem{remark}[theorem]{Remark}
\numberwithin{equation}{section}
\newcommand{\nrm}[1]{\Vert#1\Vert}
\newcommand{\abs}[1]{\vert#1\vert}
\newcommand{\brk}[1]{\langle#1\rangle}
\newcommand{\set}[1]{\{#1\}}

\newcommand{\br}[1]{\overline{#1}}

\newcommand{\dist}{\mathrm{dist}}

\renewcommand{\Re}{\mathrm{Re}}
\renewcommand{\Im}{\mathrm{Im}}

\newcommand{\aeq}{\sim}
\newcommand{\aleq}{\lesssim}

\newcommand{\lap}{\triangle}

\newcommand{\ud}{\mathrm{d}}
\newcommand{\rd}{\partial}
\newcommand{\nb}{\nabla}

\newcommand{\bb}{\Big}

\newcommand{\alp}{\alpha}
\newcommand{\bt}{\beta}
\newcommand{\gmm}{\gamma}

\newcommand{\dlt}{\delta}

\newcommand{\eps}{\epsilon}

\newcommand{\lmb}{\lambda}

\newcommand{\sgm}{\sigma}

\newcommand{\bfD}{{\bf D}}
\newcommand{\bfE}{{\bf E}}

\newcommand{\bbR}{\mathbb R}

\newcommand{\bbZ}{\mathbb Z}

\newcommand{\calC}{\mathcal C}

\newcommand{\calI}{\mathcal I}

\newcommand{\calS}{\mathcal S}

\newcommand{\covD}{\bfD}
\newcommand{\pfstep}[1]{\vskip.5em {\it #1}.}

\begin{document}

\title[]{Finite energy global well-posedness of the Chern-Simons-Higgs equations in the Coulomb gauge}
\author{Sung-Jin Oh}%
\address{Department of Mathematics, UC Berkeley, Berkeley, CA 94720}%
\email{sjoh@math.berkeley.edu}%


\begin{abstract}
In a recent paper, Selberg-Tesfahun \cite{Selberg:2012vb} proved that the abelian Chern-Simons-Higgs system (CSH) is globally well-posed for finite energy initial data under the Lorenz gauge condition. 
It has been suspected by Huh \cite{Huh:2011cp}, however, that such a result should hold in the \emph{Coulomb gauge} as well.
In this note, we give an affirmative answer to this question by first establishing low regularity local well-posededness of (CSH) in the Coulomb gauge for initial data set $(f, g) \in H^{\gmm}_{x} \times H^{\gmm-1}_{x}$ for any $\gmm > 3/4$. 
Then by conservation of energy, global well-posedness for (CSH) in the energy space $(f, g) \in H^{1}_{x} \times L^{2}_{x}$ follows rather immediately.
\end{abstract}
\maketitle
\section{Introduction}
Let $\bbR^{1+2}$ be the Minkowski space of signature $(-, +, +)$. 
The \emph{abelian Chern-Simons-Higgs system} on $\bbR^{1+2}$ takes the form
\begin{equation} \label{eq:CSH} \tag{CSH}
\left\{
\begin{aligned}
	F_{\mu \nu} = & \eps_{\mu \nu \lmb} \Im (\phi \br{\covD^{\lmb} \phi}), \\
	(- \covD_{0}^{2} + \covD_{1}^{2} + \covD_{2}^{2}) \phi = & m\phi + W(\phi),
\end{aligned}
\right.
\end{equation}
where $\phi$ is a complex-valued function (Higgs field), $A_{\mu}$ is a real-valued 1-form (Chern-Simons potential), $F = \ud A$, $\covD_{\mu} := \rd_{\mu} - i A_{\mu}$ and $\eps_{\mu \nu \lmb}$ is the standard volume form on $\bbR^{1+2}$, i.e. the unique 3-form such that $\eps_{012} = 1$.  The non-negative number $m \geq 0$ is the mass of the Higgs field and $W(\phi)$ is a self-interaction potential of the form
\begin{equation*}
	W(\phi) := \phi V'(\abs{\phi}^{2})
\end{equation*}
where $V(r)$ is a polynomial in $r$ vanishing at $0$. 

In this note, we establish local well-posedness (LWP) of \eqref{eq:CSH} in the Coulomb gauge $\rd_{1} A_{1} + \rd_{2} A_{2} = 0$ for data $(\phi, \covD_{t} \phi)(0) \in H^{\gmm}_{x} \times H^{\gmm-1}_{x}$ with $\gmm > 3/4$. We remark that this is $1/4$ away from the optimal regularity predicted by scaling considerations. Combined with the conserved energy of \eqref{eq:CSH}, global well-posedness for $\gmm=1$ (i.e., finite energy GWP) follows rather easily under appropriate conditions on $V$.
For the precise statement of the main theorems, we refer to \S \ref{subsec:mainThm}.

The \eqref{eq:CSH} model, more specifically the self-dual case $V(\abs{\phi}^{2}) = \frac{1}{16} \abs{\phi}^{2} (1-\abs{\phi}^{2})$, was first proposed by Hong-Kim-Pac \cite{hong1990multivortex} and Jackiw-Weinberg \cite{jackiw1990self} in the context of theory of planar vortex solutions\footnote{We remark that there is a vast literature, both mathematical and physical, on the study of vortex solutions to various Chern-Simons models, with particular attention to the self-dual case. It would be impossible to cover it in an adequate manner in this brief introduction; we refer the reader to the monograph \cite[Chapter 5]{MR1838682} and the references therein.}. 
Recently, the initial value problem for \eqref{eq:CSH} has received considerable attention.
In particular, after the works of Chae-Choe \cite{Chae:2002eu}, Huh \cite{Huh:2007wm}, \cite{Huh:2011cp} and Bournaveas \cite{MR2539222}, finite energy GWP of \eqref{eq:CSH} was first established by Selberg-Tesfahun \cite{Selberg:2012vb} under the Lorenz gauge condition $-\rd_{0} A_{0} +\rd_{1} A_{1} + \rd_{2} A_{2} = 0$. 
The regularity condition for local well-posedness (LWP) in the Lorenz gauge has been subsequently improved by Huh and the author \cite{Oh:2012uq} to $\gmm > 3/4$. Our results, therefore, extend the results of \cite{Selberg:2012vb} and \cite{Oh:2012uq} to the Coulomb gauge, and give a positive answer to the question raised in \cite{Huh:2011cp}, namely, whether finite energy GWP of \eqref{eq:CSH} could be proved in the Coulomb gauge.

%

Under the Coulomb gauge condition, the gauge potential components $A_{\mu}$ obey elliptic equations with quadratic terms in $\phi$ on the right-hand side. As a consequence, the leading order nonlinearity of the wave equation for $\phi$ becomes essentially cubic. 
Thanks to this feature, we are able to reach the regularity $\gmm > 3/4$ \emph{without} any use of null structure for the wave equation. This is in contrast to the Lorenz gauge setting \cite{Selberg:2012vb}, \cite{Oh:2012uq}, in which \eqref{eq:CSH} reduces to a system of wave equations with quadratic (and higher) nonlinearities, for which null structure \emph{has to be} taken into account in order to reach the same level of regularity, in view of the well-known counterexamples of Lindblad \cite{Lindblad:1996ws}. 
On the other hand, as we shall see below, we make crucial use of the special structure of the elliptic equations for $A_{\mu}$ in the Coulomb gauge.

There are two difficulties to be addressed when trying to prove LWP of \eqref{eq:CSH} in the Coulomb gauge for large, low regularity initial data.
 One is the presence\footnote{We remark that in the Lorenz gauge, this term is non-existent thanks to the gauge condition $-\rd_{t} A_{0} + \rd_{1} A_{1} + \rd_{2} A_{2} = 0$.} of the term $i \rd_{t} A_{0} \phi$ in the wave equation after expanding out all covariant derivatives, i.e.,
\begin{equation*}
	\Box \phi - m \phi = - 2 i A_{0} \rd_{t} \phi + 2 i A_{1} \rd_{1} \phi + 2 i A_{2} \rd_{2} \phi - i \rd_{t} A_{0} \phi + (-A_{0}^{2} + A_{1}^{2} + A_{2}^{2}) \phi + W(\phi),
\end{equation*}
where $\Box := -\rd_{t}^{2} + \lap$. As noted in \cite{Huh:2011cp}, the potential $\rd_{t} A_{0}$ obeys an elliptic equation which does not yield favorable estimates for regularity $\gmm \leq 1$. We refer to \cite{Huh:2011cp} for a more detailed discussion. 

The second difficulty is the coupled nature of the elliptic equations for $A_{\mu}$. More precisely, expanding out all covariant derivatives in the elliptic equations for $A_{\mu}$ in the Coulomb gauge, we arrive at the elliptic system
\begin{equation*}
\left\{
\begin{aligned}
	\lap A_{1} =& \rd_{2} \big[ \Im ( \phi \br{\rd_{t} \phi}) + A_{0} \abs{\phi}^{2} \big], \\
	\lap A_{2} =& - \rd_{1} \big[ \Im ( \phi \br{\rd_{t} \phi}) + A_{0} \abs{\phi}^{2} \big], \\
	\lap A_{0} =& 
	- \rd_{1} \big[ \Im(\phi \br{\rd_{2} \phi}) + \rd_{2} \Im(\phi \br{\rd_{1} \phi}) + A_{2} \abs{\phi}^{2} - A_{1} \abs{\phi}^{2} \big].
\end{aligned}
\right.
\end{equation*}

Due to the presence of $A_{\mu}$'s on the right-hand side, it is not immediately clear whether this system can be inverted without a smallness assumption on $\abs{\phi}^{2}$.

In fact, these two difficulties are closely related, and thus will be resolved simultaneously. Our main idea is to rearrange the wave equations to 
\begin{equation*}
\Box \phi + i \rd_{t} (A_{0} \phi) = m \phi + 2 i A_{1} \rd_{1} \phi + 2 i A_{2} \rd_{2} \phi  - i A_{0} \covD_{t} \phi + (A_{1}^{2} + A_{2}^{2}) \phi +W(\phi).
\end{equation*}
and integrate $\rd_{t}$ in $i \rd_{t} (A_{0} \phi)$ by parts in the Duhamel formula for the wave equation; we dub the resulting formula the \emph{twisted Duhamel formula} (Lemma \ref{lem:tDuhamel}). Thanks to a cancellation structure in the elliptic equation for $A_{0}$, the resulting integral turns out to obey more favorable estimates. The boundary terms from the integration-by-parts, on the other hand, allow us to directly estimate $\covD_{t} \phi= (\rd_{t} - i A_{0}) \phi$ instead of $\rd_{t} \phi$. This reveals a \emph{hierarchical structure} of the above elliptic system, which allows us to invert it without any smallness assumptions. More precisely, we may first solve for $A_{1}, A_{2}$ in terms of $\phi, \covD_{t} \phi$, and then solve for $A_{0}$ in terms of $\phi, A_{1}$ and $A_{2}$. For a more detailed discussion, we refer to \S \ref{subsec:Str4ellipticEq}.

We conclude this introduction by briefly noting the recent progress on the initial value problem for other related Chern-Simons models.
After the initial work of Huh \cite{MR2290338}, various authors such as Bournaveas-Candy-Machihara \cite{MR2975691}, \cite{Bournaveas:2013vk}, Huh \cite{MR2600515}, Huh-Oh \cite{Oh:2012uq}, Okamoto \cite{Okamoto:fk} have contributed to the understanding of the \emph{Chern-Simons-Dirac equations} (CSD). On a non-relativistic system called the \emph{Chern-Simons-Schr\"odinger equations}, an exciting progress has been recently made by Liu-Smith-Tataru \cite{Liu:2012vx}, who proved almost optimal regularity LWP for small data. This raises the question whether a similar result can be established for relativistic systems such as \eqref{eq:CSH} and (CSD).

\begin{remark} 
We would like to point out a recent preprint of Bournaveas-Candy-Machihara \cite{Bournaveas:2013vk} concerning (CSD), which appeared while this note was being prepared. It shares some similarities with the present note: Its main result is a LWP result for (CSD) in the Coulomb gauge for regularity 1/4 away from optimality, extending the previous result of Huh and the author \cite{Oh:2012uq} in the Lorenz gauge. Moreover, it avoids using any null structure as well. However, \eqref{eq:CSH} possesses additional difficulties concerning the elliptic equations for $A_{\mu}$ as discussed above, the resolution of which, the author believes, is the main contribution of the present note.
\end{remark}


\subsection{Basic properties of \eqref{eq:CSH}}
Some basic properties of \eqref{eq:CSH} are in order. First, note that when $m = 0$ and $V \equiv 0$, then \eqref{eq:CSH} is invariant under the scaling
\begin{equation*}
	(A_{\mu}, \phi)(t,x) \mapsto (\lmb^{-1} A_{\mu}, \lmb^{-1/2} \phi)(t/\lmb, x/\lmb)
\end{equation*}
for any $\lmb > 0$. The scale-invariant Sobolev space is $(A_{\mu}, \phi) \in L^{2}_{x} \times \dot{H}^{1/2}_{x}$. Heuristically, this should be the optimal space in terms of regularity for well-posedness of \eqref{eq:CSH}, even in the presence of non-trivial $m$ and $V(\phi)$.

The \eqref{eq:CSH} system possesses a conserved energy, which is of the form
\begin{equation} \label{eq:E4CSH}
	\bfE[t] :=  \frac{1}{2} \int_{\bbR^{2}}  \abs{\covD_{t} \phi}^{2} + \sum_{j = 1,2} \abs{\covD_{j}\phi}^{2}(t, x) + m \abs{\phi}^{2} + V(\abs{\phi}^{2}) \, \ud x.
\end{equation}

For sufficiently regular solution $(A_{\mu}, \phi)$ to \eqref{eq:CSH} on a time interval $I$, we have \emph{conservation of energy}, namely
\begin{equation} \label{eq:energyConsv}
	\bfE[t_{1}] = \bfE[t_{2}] \hbox{ for any } t_{1}, t_{2} \in I,
\end{equation}
which may be easily justify by differentiating \eqref{eq:E4CSH} in $t$.
%
%

Given a real-valued function $\chi$ on $\bbR^{1+2}$, the corresponding \emph{gauge transform} of $(A_{\mu}, \phi)$ is defined to be
\begin{equation} \label{eq:gt4CSH}
	(A_{\mu}, \phi) \mapsto (A_{\mu} + \rd_{\mu} \chi, e^{- i \chi} \phi).
\end{equation}

Note that \eqref{eq:CSH} is covariant under such gauge transforms. This gives rise to \emph{gauge ambiguity} for solutions to \eqref{eq:CSH}, i.e., the existence of infinitely many equivalent descriptions of a single solution, which are connected to each other by a gauge transform \eqref{eq:gt4CSH}. In order to carry out analysis, we will need to fix a specific description. In this note, we shall achieve this by prescribing an additional condition for our representative to satisfy, i.e.,
\begin{equation*}
	\rd_{1} A_{1} + \rd_{2} A_{2} = 0.
\end{equation*}

Such a choice is referred to as the \emph{Coulomb gauge}. Under this condition, \eqref{eq:CSH} leads to the following system of equations for $(A_{\mu}, \phi)$:

\begin{equation} \label{eq:CSH-C} \tag{CSH-Coulomb}
\left\{
\begin{aligned}
	\lap A_{1} =& \rd_{2} \Im ( \phi \br{\covD_{t} \phi}), \\
	\lap A_{2} =& - \rd_{1} \Im ( \phi \br{\covD_{t} \phi}), \\
	\lap A_{0} =& 
	- \rd_{1} \Im(\phi \br{\covD_{2} \phi}) + \rd_{2} \Im(\phi \br{\covD_{1} \phi}) \\
	\Box \phi + i \rd_{t} (A_{0} \phi) =&
	 m \phi + 2 i A_{1} \rd_{1} \phi + 2 i A_{2} \rd_{2} \phi  \\
	 &- i A_{0} \covD_{t} \phi + (A_{1}^{2} + A_{2}^{2}) \phi +W(\phi).
\end{aligned}
\right.
\end{equation}

\subsection{Main results} \label{subsec:mainThm}
In the rest of this note, we shall be concerned with the initial value problem for \eqref{eq:CSH-C}. For $\gmm \geq 1/2$, a triple $(f, g, a_{i})$ of complex-valued functions $f, g$ and a real-valued 1-form $a_{i}$ on $\bbR^{2}$ is a \emph{Coulomb} $\gmm$-\emph{initial data set} of \eqref{eq:CSH} if
\begin{enumerate}
\item $(f, g) \in H^{\gmm}_{x} \times H^{\gmm-1}_{x}$, and
\item The 1-form $a_{i}$ is the unique solution in $\dot{H}^{\gmm'}_{x}$ to the div-curl system
\begin{align} 
	\rd_{1} a_{1} + \rd_{2} a_{2} = & 0,  \label{eq:coulomb4id} \\
	\rd_{1} a_{2} - \rd_{2} a_{1} = & \Im(f \br{g}). \label{eq:const4CSH}
\end{align}
\end{enumerate}

The first equation \eqref{eq:coulomb4id} is the Coulomb gauge condition on $a_{i}$, whereas \eqref{eq:const4CSH} is the \emph{constraint equation} imposed by \eqref{eq:CSH}.

\begin{remark} 
Note that we do not lose much generality in restricting our attention to the Coulomb gauge, in the sense that any initial data of \eqref{eq:CSH} can be transformed into the Coulomb gauge by solving a Poisson equation. More precisely, a general initial data set of \eqref{eq:CSH} is a triple $(f, g, a_{i})$, where $a_{i}$ satisfies only the constraint equation \eqref{eq:const4CSH}. Then performing a gauge transform by $\chi$, where $\chi$ is obtained by solving
\begin{equation*}
	\lap \chi = - \rd_{1} a_{1} - \rd_{2} a_{2},
\end{equation*}
we arrive at a gauge-equivalent Coulomb initial data set. That $\chi$ satisfies a \emph{linear} elliptic PDE essentially comes from the fact that that the gauge group of \eqref{eq:CSH} is abelian, viz. $U(1)$.
\end{remark}

We have not yet specified the regularity $\gmm'$ of $a_{i}$. In this note, we will set $\gmm' = 1/2$; this is because for $\gmm > 3/4$, which is the range of regularity we consider here, condition (1) is enough to guarantee the existence of a unique solution in $\dot{H}^{1/2}_{x}$. Indeed, we have the following lemma.
\begin{lemma} 
Given any complex-valued $(f, g) \in H^{\gmm}_{x} \times H^{\gmm-1}_{x}$ for $\gmm > 3/4$, there exists a unique solution $a_{i}$ to the system \eqref{eq:coulomb4id}--\eqref{eq:const4CSH} in $\dot{H}^{1/2}_{x}$.
\end{lemma}
\begin{proof} 
Borrowing \eqref{eq:abstractEst4A:1} from Lemma \ref{lem:abstractEst4A:1} (to be proved later), we see that $\Im(f \overline{g}) \in \dot{H}^{-1/2}_{x}$. It is then easy to check that
\begin{equation*}
	(a_{1}, a_{2}) = ((-\lap)^{-1} \rd_{2} \Im(f \overline{g}), - (-\lap)^{-1} \rd_{1} \Im(f \overline{g})) 
\end{equation*}
is the unique solution to \eqref{eq:coulomb4id} and \eqref{eq:const4CSH} in $\dot{H}^{1/2}_{x}$. \qedhere
\end{proof}

By a solution to the corresponding initial value problem (IVP), we mean a solution $(A_{\mu}, \phi)$ to \eqref{eq:CSH-C} such that
\begin{equation*}
	A_{i} \vert_{t=0} = a_{i}, \quad (\phi, \covD_{t} \phi) \vert_{t=0} = (f, g).
\end{equation*}

We are now ready to state our first main theorem on the LWP of \eqref{eq:CSH-C}.
\begin{theorem}[Local well-posedness] \label{thm:lwp4CSH}
For $3/4 < \gmm \leq 1$ and a polynomial $V(r)$ such that\footnote{When $\gmm = 1$, we may let $V$ be of any degree.} $\deg V < 1+\frac{1}{1-\gmm}$ and $V(0) = 0$, \eqref{eq:CSH-C} is locally well-posed for data with regularity $H^{\gmm}_{x}$. More precisely, given a Coulomb $\gmm$-initial data set $(a_{i}, f, g)$ of \eqref{eq:CSH}, there exists $T = T(m, \nrm{(f,g)}_{H^{\gmm}_{x} \times H^{\gmm-1}_{x}}, \deg V) > 0$ and a solution $(A_{\mu}, \phi)$ to the IVP for \eqref{eq:CSH-C} on $(-T, T) \times \bbR^{2}$ such that
\begin{equation*}
	\nrm{(\phi, \covD_{t} \phi)}_{S^{\gmm}(I) \times S^{\gmm-1}(I)} \aleq \nrm{(f,g)}_{H^{\gmm}_{x} \times H^{\gmm-1}_{x}},
\end{equation*}
where $I := (-T, T)$ and $S^{\gmm}(I)$ is a function space defined in \S \ref{subsec:ftnSpace} such that $S^{\gmm}(I) \subset C_{t}(I; H^{\gmm}_{x})$. Uniqueness of $(A_{\mu}, \phi)$ holds in the space
\begin{equation*}
	(A_{\mu}, \phi, \covD_{t} \phi) \in C_{t} (I; \dot{H}^{1/2}_{x}) \times S^{\gmm}(I) \times S^{\gmm-1}(I).
\end{equation*}

Finally, smooth dependence on initial data and persistence of regularity hold.
\end{theorem}

\begin{remark} 
From the above control of $\phi$ and $\covD_{t} \phi$, we may obtain estimates for $A_{\mu}$ by inverting the elliptic equations; it turns out that $A_{\mu}$ has better regularity than simply $C_{t} (I; \dot{H}^{1/2}_{x})$. See Section \ref{sec:est4A} for the necessary structure and multilinear estimates.
\end{remark}

As a simple consequence of Theorem \ref{thm:lwp4CSH} and conservation of the energy $\bfE(t)$, we obtain GWP of \eqref{eq:CSH-C} in the energy class $H^{1}_{x} \times L^{2}_{x}$.
\begin{theorem}[Finite energy global well-posedness] \label{thm:gwp4CSH}
Let $V(r)$ be a polynomial such that $V(0) = 0$ and satisfies
\begin{equation*}
	V(r) \geq - \alp^{2} r
\end{equation*}
for some $\alp \geq 0$. Then \eqref{eq:CSH-C} is globally well-posed for data with regularity $H^{1}_{x}$, i.e., the solution given by Theorem \ref{thm:lwp4CSH} exists for all time.
\end{theorem}


%

\subsection{Organization of the paper}
We will begin in Section \ref{sec:preliminaries} by setting up notations and introducing basic tools. In particular, we will present a simple version of the \emph{twisted Duhamel formula}, which is our main analytic ingredient.
In Section \ref{sec:est4A}, we will give a discussion of the hierarchical structure of the elliptic equations for $A_{\mu}$, and prove multilinear estimates for controlling $A_{\mu}$.
In Section \ref{sec:est4phi}, we will discuss the structure of the wave equation for $\phi$ in view of the twisted Duhamel formula, and also prove multilinear estimates for controlling the nonlinear terms.
Equipped with these preparations, we will finally give proofs of Theorems \ref{thm:lwp4CSH} and \ref{thm:gwp4CSH} in Section \ref{sec:pf4mainThm}, by setting up a Picard iteration scheme using the twisted Duhamel formula.

\section{Preliminaries} \label{sec:preliminaries}
\subsection{Notations}
\begin{itemize}
\item We will employ the index notation. Greek indices will run over $0, 1, 2$, whereas latin indices will run over only the spatial indices $1, 2$. Repeated upper and lower indices will be summed up.
\item For $k \in \bbZ$, we denote $k^{+} := \max \set{k, 0}$.
\item All functions spaces over $\bbR^{2}$ will be marked with a subscript in $x$, e.g., $L^{r}_{x}$. 
\item The space of all Schwartz functions on $\bbR^{2}$ will be denoted by $\calS_{x}$. 
\item For $\gmm \in \bbR$, $H^{\gmm}_{x}$ and $\dot{H}^{\gmm}_{x}$ will denote the inhomogeneous and homogenous $L^{2}$ Sobolev space of order $\gmm$ on $\bbR^{2}$, respectively, with norms
\begin{equation*}
	\nrm{\varphi}_{H^{\gmm}_{x}} := \nrm{\brk{\nb}^{\gmm} \varphi}_{L^{2}_{x}}, \quad
	\nrm{\varphi}_{H^{\gmm}_{x}} := \nrm{\abs{\nb}^{\gmm} \varphi}_{L^{2}_{x}},
\end{equation*}
where $\abs{\nb} := \sqrt{-\lap}$ and $\brk{\nb} := \sqrt{1-\lap}$.
\item Given a normed vector space $X$ over $\bbR^{2}$ and an interval $I \subset \bbR$, we define the space $L^{q}_{t} (I;X)$ for functions on $I \times \bbR^{2}$ by $\nrm{\nrm{\varphi(t,\cdot)}_{X}}_{L^{q}_{t}}$. 
\item The spatial Fourier transform will be denoted by $\widehat{\varphi}(\xi) := \int e^{2 \pi i x \cdot \xi} \varphi(x) \, \ud x$. 
\item By $A \aleq B$, we mean there exists an implicit constant $C > 0$ such that $A \leq C B$. Dependence of $C$ will be specified by a subscript or in the form $C = C(\cdot)$. $A \aeq B$ will mean $A \aleq B$ and $B \aleq A$.
\end{itemize}

\subsection{Basic inequalities}
We assume that the reader is familiar with the standard inequalities, such as Sobolev, interpolation and Gagliardo-Nirenberg interpolation; see \cite[Appendix A]{MR2233925} for a reference.

We will make use of the following version of the Bernstein inequality.
\begin{lemma}[Bernstein inequality] \label{lem:bernstein}
Let $2 \leq p \leq \infty$ and $B \subset \bbR^{2}$. Then for $\varphi \in \calS_{x}$ with its Fourier support in $B$, i.e., $\mathrm{supp}\, \hat{\varphi} \subset B$, we have
\begin{equation*}
	\nrm{\varphi}_{L^{p}_{x}} \aleq \abs{B}^{\frac{1}{2}-\frac{1}{p}} \nrm{\varphi}_{L^{2}_{x}}.
\end{equation*}
\end{lemma}
\begin{proof} 
This is an easy consequence of Hausdorff-Young, H\"older and Plancherel. \qedhere 
\end{proof}

We will also use the following product rule for Sobolev norms (both homogeneous and inhomogeneous). For a proof, we refer the reader to \cite{DAncona:2012ke}.
\begin{lemma}[Sobolev product rule] \label{lem:SobProd}
Let $\bt_{0}, \bt_{1}, \bt_{2} \in \bbR$ satisfy
\begin{equation*}
	\bt_{0} + \bt_{1} + \bt_{2} = 1, \quad 
	\max \set{\bt_{0}, \bt_{1}, \bt_{2}} < 1.
\end{equation*}

Then for every $\varphi^{1}, \varphi^{2} \in \calS(\bbR^{2})$, we have
\begin{equation} \label{eq:SobProd}
	\nrm{\varphi^{1} \varphi^{2}}_{\dot{H}^{-\bt_{0}}_{x}} 
	\aleq_{\bt_{0}, \bt_{1}, \bt_{2}} \nrm{\varphi^{1}}_{\dot{H}^{\bt_{1}}_{x}} \nrm{\varphi^{2}}_{\dot{H}^{\bt_{2}}_{x}}
\end{equation}

The same estimate holds with $\dot{H}^{\bt_{j}}_{x}$ replaced by the inhomogeneous counterpart $H^{\bt_{j}}_{x}$.
\end{lemma}

\subsection{Littlewood-Paley theory}
In this paper, we will employ Littlewood-Paley theory as a basic tool to analyze multilinear expressions.
Let $\chi(\xi)$ be a smooth bump function which equals 1 on $\set{\xi: \abs{\xi} \leq 2}$ and supported in $\set{\xi: \abs{\xi} \leq 4}$. 
Given $k \in \bbZ$, define $\chi_{k}(\xi) := \chi(\xi/2^{k})$. Given a tempered distribution $\varphi \in \calS_{x}'$, we define its $k$-th Littlewood-Paley projection $P_{k}$ by
\begin{equation*}
	\widehat{P_{k} \varphi}(\xi) := (\chi_{k}(\xi) - \chi_{k-1}(\xi)) \widehat{\varphi}(\xi).
\end{equation*}

Define also $P_{\leq k} := \sum_{j \leq k} P_{j}$. We will often use the shorthand $\varphi_{k} := P_{k} \varphi$ and $\varphi_{\leq k} := P_{\leq k} \varphi$.

Below is a list of some basic properties of $P_{k}$. For a proof, we refer to \cite[Appendix A]{MR2233925}.
\begin{lemma} 
 For every $k \in \bbZ$ and $\varphi \in \calS_{x}$, the following statements hold.
\begin{enumerate}
\item ($L^{p}$ boundedness) For $1 \leq p \leq \infty$, we have $\nrm{P_{\leq k} \varphi}_{L^{p}_{x}} \aleq \nrm{\varphi}_{L^{p}_{x}}$, $\nrm{P_{k} \varphi}_{L^{p}_{x}} \aleq \nrm{\varphi}_{L^{p}_{x}}$.
\item (Finite band property) For $1 \leq p \leq \infty$, we have 
\begin{equation*}
\nrm{\abs{\nb} P_{k}\varphi}_{L^{p}_{x}} \aeq 2^{k} \nrm{P_{k} \varphi}_{L^{p}_{x}} \hbox{ and }
\nrm{\brk{\nb} P_{k}\varphi}_{L^{p}_{x}} \aeq 2^{k^{+}} \nrm{P_{k} \varphi}_{L^{p}_{x}}
\end{equation*}\item (Littlewood-Paley square function estimate) For $1 < p < \infty$, we have 
\begin{equation*}
\nrm{(\sum_{k \in \bbZ} \abs{P_{k} \varphi}^{2} )^{1/2}}_{L^{p}_{x}} \aeq \nrm{\varphi}_{L^{p}_{x}}.
\end{equation*}
\end{enumerate}
\end{lemma}

Given a product $\varphi^{1} \varphi^{2}$, consider the decomposition
\begin{equation*}
\varphi^{1} \varphi^{2} = \sum_{(k_{1}, k_{2}, k_{3}) \in \bbZ^{3}} P_{k_{0}}(\varphi^{1}_{k_{1}} \varphi^{2}_{k_{2}}).
\end{equation*}

We call $k_{0}$, $k_{1}$ and $k_{2}$ the frequency of the output, first input and second input, respectively. 
It turns out that not all combinations of $(k_{0}, k_{1}, k_{2})$ give rise to a non-zero summand. Define
\begin{align*}
\mathrm{LH} :=& \set{(k_{0}, k_{1}, k_{2}) \in \bbZ^{3} : k_{1} \leq k_{2}+5, \abs{k_{0} - k_{2}} \leq 5} \\
\mathrm{HL} :=& \set{(k_{0}, k_{1}, k_{2}) \in \bbZ^{3} : k_{2} \leq k_{1}+5, \abs{k_{0} - k_{1}} \leq 5} \\
\mathrm{HH} :=& \set{(k_{0}, k_{1}, k_{2}) \in \bbZ^{3} : k_{0} \leq \min \set{k_{1}, k_{2}} - 5, \abs{k_{1} - k_{2}} \leq 5} 
\end{align*}

Then by Fourier support properties, it is easy to see that
\begin{equation*}
	\sum_{(k_{1}, k_{2}, k_{3}) \in \bbZ^{3}} P_{k_{0}}(\varphi^{1}_{k_{1}} \varphi^{2}_{k_{2}}) 
	= \sum_{(k_{1}, k_{2}, k_{3}) \in \mathrm{LH} \,  \cup \mathrm{HL} \, \cup \, \mathrm{HH} } P_{k_{0}}(\varphi^{1}_{k_{1}} \varphi^{2}_{k_{2}}).
\end{equation*}

Analyzing a product with respect to the (slightly overlapping) sums over $\mathrm{LH}$, $\mathrm{HL}$ and $\mathrm{HH}$ is usually referred to as \emph{Littlewood-Paley trichotomy}. The following lemma, whose easy proof we skip, is often useful for such an analysis.
\begin{lemma}[Simple convolution bound] \label{lem:simpleConv}
Let $(b_{k})_{k \in \bbZ}$ be a non-negative sequence, and $a \geq 0$ an integer. Then for any $1 \leq p \leq \infty$, we have
\begin{equation} \label{eq:simpleConv}
\nrm{\sum_{j : \abs{j - k} \leq a} b_{j}}_{\ell^{p}_{k}(\bbZ)} \aleq_{a} \nrm{b_{j}}_{\ell^{p}_{j}(\bbZ)}
\end{equation}
\end{lemma}

\subsection{Linear theory for wave equation}
Our main analytic ingredient is a variant of the Duhamel formula for the wave equation, which we dub the \emph{twisted Duhamel formula}. Roughly speaking, this procedure allows us to avoid the problematic term $\rd_{t} (A_{0} \phi)$ at the price of estimating $\abs{\nb} (A_{0} \phi)$.  Moreover, it allows us to estimate the covariant time derivative $\covD_{t} \phi$ directly, which is useful for estimating $A_{\mu}$ in the large data case. For a more detailed discussion, we refer the reader to \S \ref{subsec:Str4ellipticEq}.

The version we present below is stated in terms of smooth functions, to illustrate the main idea in a clear setting.
\begin{lemma} [Twisted Duhamel formula] \label{lem:tDuhamel}
Consider a finite time interval $I \subset \bbR$ with $0 \in I$. Let $\phi, A_{0}, F \in C^{\infty}_{t}(I; \calS_{x})$ satisfy the equation
\begin{equation} \label{eq:rnWave}
\Box \phi + i \rd_{t} (A_{0} \phi) = F.
\end{equation}
where $\covD_{t} := \rd_{t} - i A_{0}$. Then the following \emph{twisted Duhamel formula} holds:
\begin{equation} \label{eq:tDuhamel:1}
\begin{aligned}
	\phi(t,x) 
	= &\cos t \abs{\nb} \phi(0, x) + \frac{\sin t \abs{\nb}}{\abs{\nb}} \covD_{t} \phi(0, x) \\
	& - \int_{0}^{t} \frac{\sin (t-t') \abs{\nb}}{\abs{\nb}} F(t', x) \, \ud t' + i \int_{0}^{t} \cos (t-t') \abs{\nb} (A_{0} \phi)(t', x) \, \ud t'.
\end{aligned}
\end{equation}

For $\covD_{t} \phi$, the following formula holds:
\begin{equation} \label{eq:tDuhamel:2}
\begin{aligned}
	\covD_{t} \phi(t,x) 
	= &- \sin t \abs{\nb} ( \abs{\nb}  \phi)(0, x) + \cos t \abs{\nb} \covD_{t} \phi(0, x) \\
	& - \int_{0}^{t} \cos (t-t') \abs{\nb} F(t', x) \, \ud t' - i \int_{0}^{t}  \sin (t-t') \abs{\nb} \big( \abs{\nb} (A_{0} \phi) \big) (t', x) \, \ud t'.
\end{aligned}
\end{equation}
\end{lemma}

%
%
%

\begin{proof} 
We begin with the standard Duhamel formula for the d'Alembertian:
\begin{align*}
	\phi(t,x) 
	= &\cos t \abs{\nb} \phi(0, x) + \frac{\sin t \abs{\nb}}{\abs{\nb}} \rd_{t} \phi(0, x) \\
	& - \int_{0}^{t} \frac{\sin (t-t') \abs{\nb}}{\abs{\nb}} F(t', x) \, \ud t' + i \int_{0}^{t} \frac{\sin (t-t') \abs{\nb}}{\abs{\nb}} \rd_{t'}(A_{0} \phi)(t', x) \, \ud t'
\end{align*}

Integrating the last term by parts (which is easily justified for smooth objects),
\begin{align*}
 i \int_{0}^{t} \frac{\sin (t-t') \abs{\nb}}{\abs{\nb}} \rd_{t'}(A_{0} \phi)(t',x) \, \ud t' 
= &  - i \frac{\sin t \abs{\nb}}{\abs{\nb}} A_{0} \phi (0, x) + i \int_{0}^{t} \cos(t - t') \abs{\nb} (A_{0} \phi)(t', x) \, \ud t'.
\end{align*}

Recalling that $\rd_{t} - i A_{0} = \covD_{t}$, we obtain \eqref{eq:tDuhamel:1}. To proceed, take $\rd_{t}$ of both sides. From
\begin{equation*}
i \rd_{t} \int_{0}^{t} \cos (t-t') \abs{\nb} (A_{0} \phi)(t', x) \, \ud t' = i A_{0} \phi(t,x) - i \int_{0}^{t}  \sin (t-t') \abs{\nb} \big( \abs{\nb} (A_{0} \phi) \big) (t', x) \, \ud t',
\end{equation*}
the desired identity \eqref{eq:tDuhamel:2} follows.
\end{proof}

\begin{remark} 
The benefits of the twisted Duhamel formula do not come without a price; indeed, we need a good control of $A_{0}$ in order to be able to apply this in a Picard iteration setting. Luckily, this is affordable for our system \eqref{eq:CSH-C} thanks to the special structure of the elliptic equation for $A_{0}$. See Lemma \ref{lem:nullStr4A0} and Proposition \ref{prop:tDuhamelEnergy} for more details.
\end{remark}

As a technical tool, we need a well-known refinement of the Strichartz estimate due to Klainerman-Tataru \cite{Klainerman:1999do}. To state the estimate in the form we use, we shall make some definitions. 

Given $\ell, k \in \bbZ$ with $\ell < k$, define $\calC_{\ell, k}$ to be a finitely overlapping cover\footnote{We may arrange so that the number of overlaps is uniformly bounded in $\ell, k$.} of $\set{\xi \in \bbR^{2}: 2^{k-2} \leq \abs{\xi} \leq 2^{k+2}}$ by balls of radius $2^{\ell}$. Note that the number of such balls is $\aleq 2^{2(k-\ell)}$. Let $\set{\chi_{c}}_{c \in \calC_{\ell, k}}$ be a smooth partition of unity, where each $\chi_{c}$ is supported in $c \in \calC_{\ell, k}$. 
Then define $P_{c}$ to be the Fourier multiplier with symbol $\chi_{c}$. 
Via Plancherel, it is easy to verify that $\set{P_{c}}_{c \in \calC_{\ell, k}}$ is almost orthogonal in $L^{2}_{x}$, i.e., 
\begin{equation}
\sum_{c \in \calC_{\ell, k}} \nrm{P_{c} \varphi_{k}}^{2}_{L^{2}_{x}} 
= \nrm{(\sum_{c \in \calC_{\ell, k}} \abs{P_{c} \varphi_{k}}^{2} )^{1/2} }_{L^{2}_{x}}^{2}
\aeq \nrm{\varphi_{k}}_{L^{2}_{x}}^{2}
\end{equation}

As a convention, we set $\calC_{k, k}$ to be the singleton $\set{c_{k}}$ where $\chi_{c_{k}}$ is simply equal to 1 on $\set{\xi \in \bbR^{2}: 2^{k-2} \leq \abs{\xi} \leq 2^{k+2}}$.

We are now ready state the Klainerman-Tataru refinement of Strichartz inequality.

\begin{lemma}[Klainerman-Tataru refinement of Strichartz] \label{lem:KTStr}
Consider $\ell, k \in \bbZ$ and $q, r \in \bbR$ such that $\ell \leq k$ and $(q,r)$ is Strichartz-admissible, i.e.,
\begin{equation} \label{eq:StrAdmissible}
	2 \leq q, r \leq \infty, \quad
	\frac{2}{q} \leq \frac{1}{2} - \frac{1}{r}.
\end{equation}

Then for any sign $\pm$ and $f \in \calS_{x}$, we have
\begin{align} 
	\nrm{(\sum_{c \in \calC_{\ell, k}} \abs{P_{c} (e^{\pm t \abs{\nb}} f_{k} )}^{2} )^{1/2} }_{L^{q}_{t} L^{r}_{x}} 
	\aleq 2^{(1 -\frac{2}{q} - \frac{2}{r}) (\ell- k)} 2^{(1 - \frac{1}{q} - \frac{2}{r})k} \nrm{f_{k}}_{L^{2}_{x}} \label{eq:KTStr} 
\end{align}
\end{lemma}

For a proof, we refer to \cite[Appendix A]{Klainerman:1999do}. Note that, by our convention, the case $\ell = k$ corresponds to the usual Strichartz estimate. 

\begin{remark} 
We remark that the Klainerman-Tataru refinement of the Strichartz inequality was also used in \cite{Bournaveas:2013vk} for the Chern-Simons-Dirac system, in order to handle certain high-high interactions when inverting elliptic equations for $A_{\mu}$.
\end{remark}

\subsection{Function space for wave equation} \label{subsec:ftnSpace}
Let $I \subset \bbR$ be a finite interval. For $\varphi \in C^{\infty}_{t}(I; \calS_{x})$, we define
\begin{equation} \label{eq:Def4S}
\nrm{\varphi}^{2}_{S^{0}_{k}(I)} := \nrm{\varphi}^{2}_{L^{\infty}_{t} (I;L^{2}_{x})} + \sup_{\ell \leq k} \, 2^{-(\ell- k)} 2^{-(3/2) k}  \nrm{( \sum_{c \in \calC_{\ell, k}} \abs{P_{c} \varphi}^{2} )^{1/2}}^{2}_{L^{4}_{t} (I; L^{\infty}_{x})}.
\end{equation}

The $S^{0}_{k}$ norm will be the basic dyadic building block for the space in which we will prove LWP of \eqref{eq:CSH-C}. It controls all the Klainerman-Tataru-Strichartz-type norms, as the following lemma shows.
\begin{lemma} \label{lem:basicProp4Sk}
Let $I \subset \bbR$ be a finite interval. Then for every $\varphi \in C^{\infty}(I ; \calS_{x})$, we have
\begin{equation} \label{eq:KTStr4Sk}
	\sup_{(q,r) \in \mathrm{Str}} \sup_{\ell \leq k} \, 2^{-(1 -\frac{2}{q} - \frac{2}{r}) (\ell- k)} 2^{-(1-\frac{1}{q}-\frac{2}{r}) k}  \nrm{( \sum_{c \in \calC_{\ell, k}} \abs{P_{c} \varphi}^{2} )^{1/2}}_{L^{q}_{t} (I; L^{r}_{x})} \aleq \nrm{\varphi}_{S^{0}_{k}(I)}.
\end{equation}
where $\mathrm{Str}$ is the set of all Strichartz admissible pairs, i.e.,
\begin{equation*}
 \mathrm{Str} := \set{(q,r) \in \bbR^{2} : (q,r) \hbox{ satisfies } \eqref{eq:StrAdmissible}}.
\end{equation*}
\end{lemma}
\begin{proof} 
 By interpolation, it suffices to prove the $L^{\infty}_{t} L^{2}_{x}$ and $L^{\infty}_{t,x}$ estimates, i.e.,
 \begin{equation*}
\sup_{\ell \leq k} \nrm{( \sum_{c \in \calC_{\ell, k}} \abs{P_{c} \varphi}^{2} )^{1/2}}_{L^{\infty}_{t} (I; L^{2}_{x})} \aleq \nrm{\varphi}_{S^{0}_{k}(I)}, \quad
\sup_{\ell \leq k} 2^{-\ell} \nrm{( \sum_{c \in \calC_{\ell, k}} \abs{P_{c} \varphi}^{2} )^{1/2}}_{L^{\infty}_{t} (I; L^{\infty}_{x})} \aleq \nrm{\varphi}_{S^{0}_{k}(I)}.
\end{equation*}

These estimates follow easily from Bernstein and $L^{2}$ almost orthogonality of $P_{c}$'s. \qedhere
\end{proof}


Given $\gmm \in \bbR$ and a finite interval $I \subset \bbR$, we define the inhomogeneous norm $S^{\gmm}(I)$ by
\begin{equation} \label{eq:Def4Sgmm}
	\nrm{\varphi}_{S^{\gmm}(I)} := \bb( \sum_{k} 2^{2 \gmm k^{+}} \nrm{\varphi_{k}}_{S^{0}_{k}(I)}^{2} \bb)^{1/2}
\end{equation}
and define the space $S^{\gmm}$ to be the completion of $C_{t}^{\infty} (I; \calS_{x})$ under this norm.

Some basic properties of the space $S^{\gmm}$ are in order.
\begin{lemma} \label{lem:basicProps4S}
Let $\gmm \in \bbR$ and $I \subset \bbR$ a finite interval. Then the following statements hold.
\begin{enumerate}
\item $S^{\gmm}(I)$ is a Banach space which imbeds into $C_{t} (I; H^{\gmm}_{x})$, i.e., $ S^{\gmm}(I) \subset C_{t} (I;H^{\gmm}_{x})$.
\item We have $S^{\gmm}(I) \subset L^{q}_{t} (I; L^{r}_{x})$ for $(q,r) \in \mathrm{Str}$ such that
\begin{equation*}
	1 - \frac{1}{q} - \frac{2}{r} \leq \gmm \quad \hbox{ and } \quad 
	(1 - \frac{1}{q} - \frac{2}{r}, r) \neq (\gmm, \infty).
\end{equation*}	
\item Let $0 < T \leq T_{0}$. Then for every $\varphi \in S^{\gmm}(-T_{0}, T_{0})$, $\nrm{\varphi}_{S^{\gmm}(-T, T)}$ is continuous in $T \in (0, T_{0}]$ and 
\begin{equation} \label{eq:contAt0ForS}
	\limsup_{T \to 0} \nrm{\varphi}_{S^{\gmm}(-T, T)} \aleq_{\gmm} \nrm{\varphi \vert_{t=0}}_{H^{\gmm}_{x}}.
\end{equation}
\end{enumerate}
\end{lemma}

\begin{proof} 
The first and second statements are simple consequences of the Littlewood-Paley square function estimate. For the third statement, note that it suffices to consider $\varphi \in C^{\infty}_{t}((-T_{0}, T_{0}); \calS_{x})$ by approximation. 
We claim that
\begin{equation} \label{eq:basicProps4S:pf}
\bb( \sum_{k \in \bbZ} \sup_{\ell \leq k} \, 2^{-(\ell- k)} 2^{-(3/2) k} s^{2 \gmm k^{+}} \nrm{( \sum_{c \in \calC_{\ell, k}} \abs{P_{c} \varphi_{k}}^{2} )^{1/2}}_{L^{4}_{t} (I; L^{\infty}_{x})}^{2} \bb)^{1/2}
\end{equation}
goes to zero as $\abs{I} \to 0$. Indeed, by \eqref{eq:lowFreqWEst:pf} (to be established later) and Bernstein (Lemma \ref{lem:bernstein}), we have (for some $\eps > 0$)
\begin{equation*}
\eqref{eq:basicProps4S:pf} \aleq \sum_{k \in \bbZ} 2^{\gmm k^{+}}\nrm{\varphi_{k}}_{L^{4}_{t} (I;L^{2}_{x})} \aleq \nrm{\varphi}_{L^{4}_{t} (I;H^{\gmm+\eps}_{x})} + \nrm{\varphi}_{L^{4}_{t} (I; L^{2-\eps}_{x})} \to 0 \hbox{ as } \abs{I} \to 0.
\end{equation*}

Recalling the definitions of $S^{0}_{k}$ and $S^{\gmm}$, the desired continuity statement now follows. \qedhere
\end{proof}

The space $S^{\gmm}$ is the main function space in which we will carry out a Picard iteration scheme. As discussed in the introduction, we will use the twisted Duhamel formula (Lemma \ref{lem:tDuhamel}) instead of the usual Duhamel formula and derive an `energy estimate', i.e., an estimate of $(\phi, \covD_{t} \phi)$ in $S^{\gmm} \times S^{\gmm-1}$ in terms of the initial data, $F$ and $A_{0}$. However, due to the presence of $\phi$ on the right-hand side of \eqref{eq:tDuhamel:1} and \eqref{eq:tDuhamel:2}, we need to develop more machinery in order to prove such a result. See Proposition \ref{prop:tDuhamelEnergy} for the final product.

In the remainder of this subsection, we collect some estimates which will be useful for proving the desired energy estimate. The following two lemmas show that a solution to the free wave equation with initial data in $H^{\gmm}_{x} \times H^{\gmm-1}_{x}$ belongs to $S^{\gmm}$. 
\begin{lemma} \label{lem:homEnergy:1}
Let $\gmm \in \bbR$ and $I \subset \bbR$ a finite interval. Then for $f \in H^{\gmm}_{x}$, 
\begin{equation*}
\cos t \abs{\nb} f, \, \sin t \abs{\nb} f \in S^{\gmm}(I),
\end{equation*}
and the following estimates hold.
\begin{gather} 
	\nrm{\cos t \abs{\nb} f}_{S^{\gmm}(I)} + \nrm{\sin t \abs{\nb} f}_{S^{\gmm}(I)} \aleq \nrm{f}_{H^{\gmm}_{x}} 
\end{gather}
\end{lemma}
\begin{proof} 
This is an obvious consequence of Lemma \ref{lem:KTStr}, by an approximation argument. 
\end{proof}

\begin{lemma} \label{lem:homEnergy:2}
Let $\gmm \in \bbR$ and $I \subset \bbR$ an interval such that $0 \in I$ and $\abs{I} \leq 1$. Then for $g \in H^{\gmm-1}$, 
\begin{equation*}
	\frac{\sin t \abs{\nb}}{\abs{\nb}} g \in S^{\gmm}(I)
\end{equation*}
and the following estimate holds.
\begin{equation*}
	\nrm{\frac{\sin t \abs{\nb}}{\abs{\nb}} g}_{S^{\gmm}(I)} \aleq \nrm{g}_{H^{\gmm-1}_{x}}.
\end{equation*}
\end{lemma}

The operator $\abs{\nb}^{-1}$ is unfavorable for low frequencies; we will basically integrate in time to overcome this. The (proof of the) following lemma makes this idea more precise.
\begin{lemma} \label{lem:lowFreqWEst}
Let $k \in \bbZ$ and $I \subset \bbR$ satsify $k \leq 0$ and $\abs{I} \leq 1$. Then for $\varphi \in C^{\infty}_{t}(I; \calS_{x})$, 
\begin{equation} \label{eq:lowFreqWEst}
\nrm{\varphi_{k}}_{S^{0}_{k}}  \aleq \nrm{\varphi_{k}}_{L^{\infty}_{t}(I; L^{2}_{x})}
\end{equation}
\end{lemma}

\begin{proof} 
For $\ell \leq k \leq 0$ and $\abs{I} \leq 1$, we claim that
\begin{equation} \label{eq:lowFreqWEst:pf}
 \nrm{(\sum_{c \in \calC_{\ell, k}} \abs{P_{c} \varphi_{k}}^{2} )^{1/2}}_{L^{4}_{t} (I; L^{\infty}_{x})}
\aleq 2^{(1/2) (\ell- k)} 2^{(3/4)k}\nrm{\varphi_{k}}_{L^{4}_{t} (I; L^{2}_{x})}
\end{equation}

Then \eqref{eq:lowFreqWEst} would follow by H\"older in time (using $\abs{I} \leq 1$).

Let $t \in I$. Then by Bernstein (Lemma \ref{lem:bernstein}) and $L^{2}$ almost orthogonality of $P_{c}$'s,
\begin{align*}
\nrm{(\sum_{c \in \calC_{\ell, k}} \abs{P_{c} \varphi_{k}(t)}^{2} )^{1/2} }_{L^{\infty}_{x}}
\aleq &(\sum_{c \in \calC_{\ell, k}} 2^{\ell} \nrm{P_{c} \varphi_{k}(t)}^{2}_{L^{2}_{x}} )^{1/2} 
\aleq 2^{(1/2) (\ell- k)} 2^{(3/4) k} \nrm{\varphi_{k}(t)}_{L^{2}_{x}}.
\end{align*}

Taking the $L^{4}_{t}$ norm, \eqref{eq:lowFreqWEst:pf} follows. \qedhere
\end{proof}

We are ready to prove Lemma \ref{lem:homEnergy:2}
\begin{proof} [Proof of Lemma \ref{lem:homEnergy:2}]
By an approximation argument, it suffices to consider $g \in \calS_{x}$. We wish to establish
\begin{equation} \label{eq:homEnergy:pf:1}
\bb( \sum_{k \in \bbZ} 2^{2 \gmm k^{+}} \nrm{\frac{\sin t \abs{\nb}}{\abs{\nb}} g_{k}}^{2}_{S^{0}_{k}(I)} \bb)^{1/2} \aleq \nrm{g}_{H^{\gmm-1}_{x}}
\end{equation}

For $k > 0$, we have $2^{\gmm k^{+}} \nrm{\frac{\sin t \abs{\nb}}{\abs{\nb}} g_{k}}_{S^{0}_{k}(I)} \aleq 2^{(\gmm-1) k}\nrm{\sin t \abs{\nb} g_{k}}_{S^{0}_{k}(I)}$. Therefore, \eqref{eq:homEnergy:pf:1} for $k > 0$ follows from Lemma \ref{lem:homEnergy:1}.

For $k \leq 0$, we apply Lemma \ref{lem:lowFreqWEst} and use the fact that $\sup_{t \in I} \abs{t} \leq 1$ to proceed as follows:
\begin{align*}
2^{\gmm k^{+}} \nrm{\frac{\sin t \abs{\nb}}{\abs{\nb}} g_{k}}_{S^{0}_{k}(I)}
\aleq & \nrm{\frac{\sin t \abs{\nb}}{\abs{\nb}} g_{k}}_{L^{\infty}_{t} (I; L^{2}_{x})} 
\leq \nrm{\frac{\sin t \abs{\nb}}{t \abs{\nb}} g_{k}}_{L^{\infty}_{t} (I; L^{2}_{x})}
\end{align*}

Note that $\frac{\sin t \abs{\nb}}{t \abs{\nb}}$ has a symbol which is uniformly bounded in $t$; thus the right-hand side is bounded by $\aleq \nrm{g_{k}}_{L^{2}_{x}}$. Then by the $L^{2}$ almost orthogonality of $P_{k}$'s, we obtain \eqref{eq:homEnergy:pf:1} for $k \leq 0$ as well. \qedhere
\end{proof}

As a simple corollary of the above estimates, we obtain estimates for Duhamel integrals. We omit the easy proof.
\begin{corollary} \label{cor:inhomEnergy}
Let $I \subset \bbR$ be an interval such that $0 \in I$ and $\abs{I} \leq 1$. Then the following statements hold.
\begin{enumerate}
\item For $F \in L^{1}_{t} (I; H^{\gmm-1}_{x})$, 
\begin{equation*}
	\int_{0}^{t} \frac{\sin (t-t') \abs{\nb}}{\abs{\nb}} F(t') \, \ud t'  \in S^{\gmm}(I)
\end{equation*}
and the following estimate holds.
\begin{equation}
	\nrm{\int_{0}^{t} \frac{\sin (t-t') \abs{\nb}}{\abs{\nb}} F(t') \, \ud t'}_{S^{\gmm}(I)} \aleq \nrm{F}_{L^{1}_{t} (I; H^{\gmm-1}_{x})}.
\end{equation}

\item For $G \in L^{1}_{t} (I; H^{\gmm}_{x})$, 
\begin{equation*}
	\int_{0}^{t} \sin (t-t') \abs{\nb} G(t') \, \ud t', \,
	\int_{0}^{t} \cos (t-t') \abs{\nb} G(t') \, \ud t' \in S^{\gmm}(I)	
\end{equation*}
and the following estimates hold.
\begin{gather}
	\nrm{\int_{0}^{t} \sin (t-t') \abs{\nb} G(t') \, \ud t'}_{S^{\gmm}(I)} \aleq \nrm{G}_{L^{1}_{t} (I; H^{\gmm}_{x})} , \\
	\nrm{\int_{0}^{t} \cos (t-t') \abs{\nb} G(t') \, \ud t'}_{S^{\gmm}(I)} \aleq \nrm{G}_{L^{1}_{t} (I; H^{\gmm}_{x})}.
\end{gather}
\end{enumerate}
\end{corollary}

\subsection{Some bilinear estimates for $(\mathrm{HH})$ interaction}
The Klainerman-Tataru refinement of the Strichartz estimate (Lemma \ref{lem:KTStr}) offers an improvement over the usual Strichartz estimate when analyzing $(\mathrm{HH})$ interactions. In this subsection, we present two technical lemmas to make this idea more precise.

The following estimate gives an improved estimate for the $(\mathrm{HH})$ interaction in `double Strichartz' norms.
\begin{lemma} \label{lem:bilinStr}
Let $I \subset \bbR$ be a finite interval and $q, r \in \bbR$ and $\sgm \in \bbR$ satisfy
\begin{equation*}
	r < \infty, \quad
	 (q,r) \in \mathrm{Str}, \quad
	-2 + \frac{4}{r} + \frac{4}{q} < \sgm < 0.
\end{equation*}

Then for $\varphi^{1}, \varphi^{2} \in C_{t}^{\infty} (I; \calS_{x})$ and $(k_{0}, k_{1}, k_{2}) \in \mathrm{HH}$, we have
\begin{equation} \label{eq:bilinStr}
	\nrm{\abs{\nb}^{\sgm} P_{k_{0}}(\varphi^{1}_{k_{1}} \varphi^{2}_{k_{2}})}_{L^{q/2}_{t} (I; L^{r/2}_{x})} \aleq 2^{\gmm k_{1}} \nrm{\varphi^{1}_{k_{1}}}_{S^{0}_{k_{1}}(I)} 2^{\gmm k_{2}} \nrm{\varphi^{2}_{k_{2}}}_{S^{0}_{k_{2}}(I)}
\end{equation}
where $\gmm = 1 - \frac{2}{r} - \frac{1}{q}$.

\end{lemma}

For a proof, we refer to \cite[Proof of Theorem 4]{Klainerman:1999do}; we remark that although the theorem in \cite{Klainerman:1999do} is proved for homogeneous waves, the above statement may be easily read off from their proof.

We also need the following technical variant of \eqref{eq:bilinStr}. We remark that our proof below is a slight variant of the aforementioned \cite[Proof of Theorem 4]{Klainerman:1999do}.
\begin{lemma} \label{lem:bilinVarStr}
Let $\sgm > - 1/2$. Then for $\varphi^{1}, \varphi^{2} \in C_{t}^{\infty}(I, \calS_{x})$, we have
\begin{equation} \label{eq:bilinVarStr}
	\nrm{\sum_{(k_{0}, k_{1}, k_{2}) \in \mathrm{HH}} \brk{\nb}^{\sgm} P_{k_{0}} (\varphi^{1}_{k_{1}} \varphi^{2}_{k_{2}})}_{L^{2}_{t}(I; L^{2}_{x})} \aleq \nrm{\varphi^{1}}_{L^{4}_{t} (I; \dot{H}^{3/4}_{x})} \nrm{\varphi^{2}}_{S^{\gmm}}
\end{equation} 
\end{lemma}

\begin{proof} 
For notational convenience, we shall omit writing $I$.  Using triangle, we may estimate
\begin{align*}
\nrm{\sum_{(k_{0}, k_{1}, k_{2}) \in \mathrm{HH} }  \brk{\nb}^{\sgm} P_{k_{0}} (\varphi^{1}_{k_{1}} \varphi^{2}_{k_{2}})}_{L^{2}_{t,x}} 
\aleq & \sum_{k_{1}, k_{2}: \abs{k_{1} - k_{2}} \leq 5 } \bb( \sum_{k_{0} \leq \min \set{k_{1}, k_{2}}-5} 2^{\sgm k^{+}_{0}}  \nrm{P_{k_{0}}(\varphi^{1}_{k_{1}} \varphi^{2}_{k_{2}}) }_{L^{2}_{t,x}} \bb)
\end{align*}

In order to proceed, let us further decompose 
\begin{equation*}
\varphi^{1}_{k_{1}} = \sum_{c_{1} \in \calC_{k_{0}, k_{1}} } P_{c_{1}} \varphi^{1}_{k_{1}}, \quad
\varphi^{2}_{k_{2}} = \sum_{c_{2} \in \calC_{k_{0}, k_{2}} } P_{c_{2}} \varphi^{2}_{k_{2}}.
\end{equation*}

Observe that $P_{k_{0}} (P_{c_{1}} \varphi^{1}_{k_{1}} P_{c_{2}} \varphi^{2}_{k_{2}})$ is non-vacuous only if, say, $\dist(c_{1}, - c_{2}) \leq 2^{k_{0}+5}$, where $\dist (A, B) := \inf_{x \in A, y \in B} \abs{x-y}$. Moreover, given $c_{1} \in \calC_{k_{0}, k_{1}}$ there exist only finitely many $c_{2}$'s that satisfy $\dist(c_{1}, - c_{2}) \leq 2^{k_{0}+5}$, with the number bounded by an absolute constant. The same statement holds with $1$ and $2$ interchanged. Therefore, by Cauchy-Schwarz, Lemma \ref{lem:simpleConv} and H\"older, we estimate
\begin{align*}
2^{\sgm k^{+}_{0}} \nrm{P_{k_{0}}(\varphi^{1}_{k_{1}} \varphi^{2}_{k_{2}}) }_{L^{2}_{t,x}} 
= & 2^{\sgm k^{+}_{0}} \nrm{\sum_{\substack{(c_{1}, c_{2}) \in \calC_{k_{0}, k_{1}} \times \calC_{k_{0}, k_{2}} \\ \dist(c_{1}, - c_{2}) \leq 2^{k_{0}+5}}} P_{k_{0}}(P_{c_{1}} \varphi^{1}_{k_{1}} P_{c_{2}} \varphi^{2}_{k_{2}}) }_{L^{2}_{t,x}} \\
\aleq & 2^{\sgm k^{+}_{0}} \nrm{(\sum_{c_{1} \in \calC_{k_{0}, k_{1}}} \abs{P_{c_{1}} \varphi^{1}_{k_{1}}}^{2})^{1/2} \, 
						(\sum_{c_{2} \in \calC_{k_{0}, k_{2}} } \abs{P_{c_{2}} \varphi^{2}_{k_{2}}}^{2} )^{1/2} }_{L^{2}_{t,x}}  \\
\aleq & 2^{\sgm k^{+}_{0}} \nrm{(\sum_{c_{1} \in \calC_{k_{0}, k_{1}}} \abs{P_{c_{1}} \varphi^{1}_{k_{1}}}^{2})^{1/2}}_{L^{4}_{t} L^{2}_{x}} 
					\nrm{(\sum_{c_{2} \in \calC_{k_{0}, k_{2}} } \abs{P_{c_{2}} \varphi^{2}_{k_{2}}}^{2} )^{1/2} }_{L^{4}_{t} L^{\infty}_{x}} 
\end{align*}

By $L^{2}$ almost orthogonality of $P_{c}$'s, \eqref{eq:KTStr4Sk} and the condition $\abs{k_{1} - k_{2}} \leq 5$ , the last line is bounded by
\begin{align*}
	\aleq 2^{(1/2)(k_{0} - k_{1})} 2^{\sgm(k^{+}_{0} - k^{+}_{1})} 2^{(3/4) k_{1}} \nrm{\varphi^{1}_{k_{1}}}_{L^{4}_{t} L^{2}_{x}} 2^{\sgm k^{+}_{2}} \nrm{\varphi^{2}_{k_{2}}}_{S^{0}_{k_{2}}}.
\end{align*}

For $\sgm > -1/2$, we have
\begin{equation*}
	\sum_{k_{0}: k_{0} \leq \min \set{k_{1}, k_{2}} - 5} 2^{(1/2)(k_{0} - k_{1})} 2^{\sgm(k^{+}_{0} - k^{+}_{1})} \aleq 1.
\end{equation*}

Therefore, by Cauchy-Schwarz and Lemma \ref{lem:simpleConv}, the desired conclusion follows. \qedhere
\end{proof}

\section{Estimates for $A_{\mu}$} \label{sec:est4A}
From the Coulomb gauge condition, the components of the gauge potential $A_{\mu}$ satisfy an elliptic system. In \S \ref{subsec:Str4ellipticEq}, a discussion of the structure of this elliptic system will be given, highlighting the properties which will be important for the purpose of proving large data LWP (and thus finite energy GWP as well). In \S \ref{subsec:abstractEst4A}, we give multilinear estimates for estimating $A_{\mu}$. 
 
\subsection{Structure of the elliptic equations for $A_{\mu}$} \label{subsec:Str4ellipticEq}
To first approximation, the elliptic equations take the form $A = \abs{\nb}^{-1} (\phi \covD_{t,x} \phi)$. More precisely, recall from \eqref{eq:CSH-C} that $A_{1}, A_{2}$ solve
\begin{align*}
	\lap A_{1} =& \rd_{2} \Im(\phi \overline{\covD_{t} \phi}), \\
	\lap A_{2} =& - \rd_{1} \Im(\phi \overline{\covD_{t} \phi}).
\end{align*}

Thanks to Lemma \ref{lem:tDuhamel} (twisted Duhamel formula), we need not expand $\covD_{t} \phi = \rd_{t} \phi - i A_{0} \phi$. 
On the other hand, for $A_{0}$, we shall expand $\covD_{j} = \rd_{j} - i A_{j}$. Accordingly, we shall split $A_{0}$ into $A_{0,1} + A_{0,2}$  as follows:
\begin{align*}
	\lap A_{0,1} = & - \rd_{1} \Im(\phi \br{\rd_{2} \phi}) + \rd_{2} \Im(\phi \br{\rd_{1} \phi}) , \\
	\lap A_{0,2} = &- \rd_{1} (A_{2} \abs{\phi}^{2}) + \rd_{2} (A_{1} \abs{\phi}^{2}) .
\end{align*}

Two structural properties of the above system are important for us. The first is that the elliptic equation for $A_{0,1}$ possesses a special cancellation structure, namely a $Q_{12}$-type null structure. This has already been noted by \cite{Huh:2007wm}, who used Wente's inequality to estimate $A_{0}$. For us, this structure is crucial for making the twisted Duhamel's formula (Lemma \ref{lem:tDuhamel}) work; more precisely, it allows us to put $A_{0} \in \mathfrak{A}_{0}^{\gmm}$ in Proposition \ref{prop:tDuhamelEnergy}, whereas other components do not belong to $\mathfrak{A}_{0}^{\gmm}$ in general.
(For the definition of $\mathfrak{A}^{\gmm}_{0}$, see \S \ref{sec:pf4mainThm}.)

The second important property of the above system is its hierarchical structure, which allows us to invert the system for data (i.e., $\phi$ and $\covD_{t} \phi$) of any size. Indeed, note that we may first solve for $A_{1}, A_{2}$ from the knowledge of $\phi$ and $\covD_{t} \phi$, and then use those to solve for $A_{0} = A_{0,1} + A_{0,2}$. The hierarchical structure arises from the fact that we are controlling $\covD_{t} \phi$, instead of $\rd_{t} \phi$, via the twisted Duhamel formula.

\subsection{Multilinear estimates} \label{subsec:abstractEst4A}
Here we derive multilinear estimates for estimating $A_{\mu}$ from the above elliptic system. We remark that the exponents in all the estimates below are chosen so that they are `optimal' as $\gmm \to 3/4$.

We begin by stating the bilinear estimates we shall need. Our first bilinear estimate, which bounds the $\dot{H}^{\bt}_{x}$ norm of $\abs{\nb}^{-1} (\varphi^{1} \varphi^{2})$, gives us a useful control for low frequency part of the product, i.e., $k_{0} < 0$. 
\begin{lemma} \label{lem:abstractEst4A:1}
Let $3/4 < \gmm < 1$ and $I \subset \bbR$ an interval with $\abs{I} \leq 1$.
Then for $\varphi^{1}, \varphi^{2} \in C^{\infty}_{t}(I; \calS_{x})$, we have
\begin{equation}\label{eq:abstractEst4A:1} 
	\nrm{\abs{\nb}^{-1} (\varphi^{1} \varphi^{2})}_{L^{\infty}_{t} (I; \dot{H}^{\bt}_{x})} 
		\aleq \nrm{\varphi^{1}}_{S^{\gmm}(I)} \nrm{\varphi^{2}}_{S^{\gmm-1}(I)}. 
\end{equation}
for $0 < \bt \leq 2(\gmm - 1/2)$. In the case $\gmm = 1$, the same estimate holds for $0 < \bt < 1$.
\end{lemma}

Our next set of bilinear estimates take into account the Strichartz estimates and improves on the regularity and integrability in $x$ by integrating in time.
\begin{lemma} \label{lem:abstractEst4A:2}
Let $\gmm > 3/4$ and $I \subset \bbR$ an interval such that $\abs{I} \leq 1$. Then for $\varphi^{1}, \varphi^{2} \in C^{\infty}_{t} (I; \calS_{x})$, we have
\begin{align}
	\nrm{\abs{\nb}^{-1} (\varphi^{1} \varphi^{2})}_{L^{4}_{t} (I; \dot{H}^{3/4}_{x})} 
		& \aleq \nrm{\varphi^{1}}_{S^{\gmm}(I)}\nrm{\varphi^{2}}_{S^{\gmm-1}(I)}, \label{eq:abstractEst4A:2} \\
	\nrm{\abs{\nb}^{-1} (\varphi^{1} \varphi^{2})}_{L^{2}_{t} (I; L^{\infty}_{x})} 
		& \aleq \nrm{\varphi^{1}}_{S^{\gmm}(I)}\nrm{\varphi^{2}}_{S^{\gmm-1}(I)}. \label{eq:abstractEst4A:3} 
\end{align}
\end{lemma}

For estimating $A_{0}$, we will use the following bilinear estimate which makes use of the $Q_{12}$-type null structure. This may be thought of as a variant of Wente's inequality.
\begin{lemma}[Wente-type inequality] \label{lem:nullStr4A0}
Let $3/4 < \gmm < 7/4$ and $I \subset \bbR$ an interval with $\abs{I} \leq 1$. Then for $\varphi^{1}, \varphi^{2} \in C^{\infty}_{t}(I; \calS_{x})$, we have
\begin{equation} \label{eq:abstractEst4A:4}
\nrm{(-\lap)^{-1} \bb( \rd_{1} (\varphi^{1} \rd_{2} \varphi^{2}) - \rd_{2} (\varphi^{1} \rd_{1} \varphi^{2}) \bb)}_{L^{4}_{t} (I; \dot{H}^{2(\gmm-1/2)+1/4}_{x})} \aleq_{\bt} \nrm{\varphi^{1}}_{S^{\gmm}(I)} \nrm{\varphi^{2}}_{S^{\gmm-1}(I)}.
\end{equation}

More generally, for $ 3/4 \leq \bt \leq 2(\gmm-1/2)+1/4$, we have
\begin{equation} \label{eq:abstractEst4A:5}
\nrm{(-\lap)^{-1} \bb( \rd_{1} (\varphi^{1} \rd_{2} \varphi^{2}) - \rd_{2} (\varphi^{1} \rd_{1} \varphi^{2}) \bb)}_{L^{4}_{t} (I; \dot{H}^{\bt}_{x})} \aleq_{\bt} \nrm{\varphi^{1}}_{S^{\gmm}(I)} \nrm{\varphi^{2}}_{S^{\gmm-1}(I)}.
\end{equation}
\end{lemma}

We now turn to the proofs. The basic strategy for proving \eqref{eq:abstractEst4A:1}--\eqref{eq:abstractEst4A:4} is to employ the Littlewood-Paley decomposition. In particular, we will take extra care for low frequencies to ensure that the stronger $\dot{H}^{\gmm-1}_{x}$ norm (instead of $H^{\gmm-1}_{x}$) is not used.


%
%

\begin{proof}[Proof of \eqref{eq:abstractEst4A:1}]
Fix a time $t \in I$. We claim that for $0 < \bt <1$, we have
\begin{equation} \label{eq:abstractEst4A:1:pf:1}
	\nrm{\abs{\nb}^{-1} (\varphi^{1} \varphi^{2})(t)}_{\dot{H}^{\bt}_{x}} 
	\aleq \nrm{\varphi^{1}(t)}_{H^{\frac{\bt+1}{2}}_{x}} \nrm{\varphi^{2}(t)}_{H^{\frac{\bt-1}{2}}_{x}}
\end{equation}

In what follows, we shall omit writing $t$ for convenience. We divide into two case:
\pfstep{Case 1: $k_{2} \leq 0$} 
In this case, note that $\nrm{P_{\leq 0} \varphi^{2}}_{L^{2}_{x}} \aleq \nrm{\varphi^{2}}_{H^{\frac{\bt-1}{2}}_{x}}$. Then by Lemma \ref{lem:SobProd}, 
\begin{equation*}
\nrm{\varphi^{1} P_{\leq 0} \varphi^{2}}_{\dot{H}^{\bt-1}_{x}} 
\aleq \nrm{\varphi^{1}}_{\dot{H}^{\bt}_{x}} \nrm{\varphi^{2}}_{L^{2}_{x}}
 \aleq \nrm{\varphi^{1}}_{H^{\frac{\bt+1}{2}}_{x}} \nrm{\varphi^{2}}_{H^{\frac{\bt-1}{2}}_{x}}.
\end{equation*}

\pfstep{Case 2: $k_{2} > 0$} 
In this case, we have $\nrm{P_{> 0} \varphi^{2}}_{\dot{H}^{\bt-1}_{x}} \aleq \nrm{\varphi^{2}}_{H^{\bt-1}_{x}}$. Again by Lemma \ref{lem:SobProd},
\begin{equation*}
\nrm{\varphi^{1} P_{> 0} \varphi^{2}}_{\dot{H}^{\bt-1}_{x}} 
 \aleq \nrm{\varphi^{1}}_{\dot{H}^{\frac{\bt+1}{2}}_{x}} \nrm{\varphi^{2}}_{\dot{H}^{\frac{\bt-1}{2}}_{x}}
  \aleq \nrm{\varphi^{1}}_{H^{\frac{\bt+1}{2}}_{x}} \nrm{\varphi^{2}}_{H^{\frac{\bt-1}{2}}_{x}}.
\end{equation*}

From \eqref{eq:abstractEst4A:1:pf:1}, the desired estimate \eqref{eq:abstractEst4A:1} follows immediately. \qedhere
\end{proof}

%
%

\begin{proof}[Proof of \eqref{eq:abstractEst4A:2}]
Below, we shall suppress the time interval $I$. The fact that $\abs{I} \leq 1$ will be freely used to throw away any positive power of $\abs{I}$ arising from H\"older in time.

For $P_{\leq 0} (\abs{\nb}^{-1} \varphi^{1} \varphi^{2})$, the desired estimate follows from the $L^{\infty}_{t} \dot{H}^{1/2}_{x}$ estimate in \eqref{eq:abstractEst4A:1} and H\"older in time. Thus, it suffices to consider $k_{0} > 0$. We divide into three cases according to Littlewood-Paley trichotomy.

\pfstep{Case 1: (LH) interaction, $k_{0} > 0$, $k_{1} \leq k_{2} +5$ and $\abs{k_{0} - k_{2}} \leq 5$} 
By $L^{2}$ almost orthogonality of $P_{k}$'s and Lemma \ref{lem:simpleConv}, we have
\begin{align*}
\nrm{\sum_{(k_{0}, k_{1}, k_{2}) \in \mathrm{LH}, k_{0} > 0} P_{k_{0}}(\varphi^{1}_{k_{1}} \varphi^{2}_{k_{2}})}_{L^{4}_{t} \dot{H}^{-1/4}_{x}}
\aleq &\bb( \sum_{k_{0} > 0} 2^{-k_{0}/2} \nrm{\sum_{k_{1}, k_{2} : (k_{0}, k_{1}, k_{2}) \in \mathrm{LH}} P_{k_{0}}(\varphi^{1}_{k_{1}} \varphi^{2}_{k_{2}})}_{L^{4}_{t} L^{2}_{x}}^{2} \bb)^{1/2} \\
\aleq &\bb( \sum_{k_{2} > -5} 2^{-k_{2}/2} \nrm{\sum_{k_{1} \leq k_{2}+5} \varphi^{1}_{k_{1}} \varphi^{2}_{k_{2}}}^{2}_{L^{4}_{t} L^{2}_{x}} \bb)^{1/2}
\end{align*}

We estimate each summand of the above $\ell^{2}$ sum by triangle, H\"older and Strichartz as follows:
\begin{align*}
2^{-k_{2}/4} \nrm{\sum_{k_{1}: k_{1} \leq k_{2}+5}  \varphi^{1}_{k_{1}} \varphi^{2}_{k_{2}}}_{L^{4}_{t} L^{2}_{x}}
\aleq & 2^{-k_{2}/4} \sum_{k_{1}: k_{1} \leq k_{2}+5} \nrm{\varphi^{1}_{k_{1}}}_{L^{4}_{t} L^{\infty}_{x}} \nrm{\varphi^{2}_{k_{2}}}_{L^{\infty}_{t} L^{2}_{x}} \\
\aleq & \nrm{\varphi^{1}}_{S^{\gmm}} \, ( 2^{-k_{2}/4} \nrm{\varphi^{2}_{k_{2}}}_{L^{\infty}_{t} L^{2}_{x}} ) \sum_{k_{1} \leq k_{2}+5} 2^{(3/4) k_{1}} 2^{-\gmm k^{+}_{1}}
\end{align*}

As $\gmm > 3/4$, the $k_{1}$ sum is uniformly bounded by $\aleq 1$ in $k_{2}$. Thus summing up in $k_{2} > -5$,
\begin{equation*}
\nrm{\sum_{(k_{0}, k_{1}, k_{2}) \in \mathrm{LH}, k_{0} > 0} P_{k_{0}}(\varphi^{1}_{k_{1}} \varphi^{2}_{k_{2}})}_{L^{4}_{t} \dot{H}^{-1/4}_{x}}
\aleq \nrm{\varphi^{1}}_{S^{\gmm}} \nrm{\varphi^{2}}_{S^{-1/4}}
\aleq \nrm{\varphi^{1}}_{S^{\gmm}} \nrm{\varphi^{2}}_{S^{\gmm-1}},
\end{equation*}
where we have used again $\gmm > 3/4$ for the last inequality. 

\pfstep{Case 2: (HL) interaction, $k_{0} > 0$, $k_{2} \leq k_{1} +5$ and $\abs{k_{0} - k_{1}} \leq 5$}
This is more favorable than Case 1, as we are allowed to use more derivatives to control the high frequency factor. We leave the easy details to the reader.

\pfstep{Case 3: (HH) interaction, $0 < k_{0} \leq \min\set{k_{1}, k_{2}} - 5$ and $\abs{k_{1} - k_{2}} \leq 5$}
Here the fractional integration $\abs{\nb}^{-1/4}$ unfavorable and we need the full power of the bilinear Strichartz estimate Lemma \ref{lem:bilinStr}.
We begin by applying triangle to bound
\begin{align*}
\nrm{\sum_{(k_{0}, k_{1}, k_{2}) \in \mathrm{HH}, k_{0} > 0} P_{k_{0}}(\varphi^{1}_{k_{1}} \varphi^{2}_{k_{2}})}_{L^{4}_{t} \dot{H}^{-1/4}_{x}}
\aleq &\sum_{k_{1}, k_{2} > 5, \abs{k_{1} - k_{2}} \leq 5} \nrm{\sum_{0 < k_{0} \leq \min \set{k_{1}, k_{2}} - 5} \abs{\nb}^{-1/4} P_{k_{0}} (\varphi^{1}_{k_{1}} \varphi^{2}_{k_{2}})}_{L^{4}_{t} L^{2}_{x}}
\end{align*}

We then estimate each summand by Lemma \ref{lem:bilinStr} as follows:
\begin{align*}
& \nrm{\sum_{k_{0}: 0 < k_{0} \leq \min \set{k_{1}, k_{2}} - 5} \abs{\nb}^{-1/4} P_{k_{0}} (\varphi^{1}_{k_{1}} \varphi^{2}_{k_{2}})}_{L^{4}_{t} L^{2}_{x}} \\
& \quad \aleq  \sum_{k_{0}: 0 < k_{0} \leq \min \set{k_{1}, k_{2}} - 5} 2^{k_{1}/4} \nrm{\varphi^{1}_{k_{1}}}_{S^{0}_{k_{1}}} 2^{k_{2}/4} \nrm{\varphi^{2}_{k_{2}}}_{S^{0}_{k_{2}}} \\
& \quad \aleq  2^{(3/2-\gmm+\eps)k_{1}} \nrm{\varphi^{1}_{k_{1}}}_{S^{0}_{k_{1}}} 2^{(\gmm-1)k_{2}} \nrm{\varphi^{2}_{k_{2}}}_{S^{0}_{k_{2}}}.
\end{align*}
where $\eps > 0$ may be arbitrarily small. Then by Cauchy-Schwarz and Lemma \ref{lem:simpleConv}, we obtain
\begin{align*}
\nrm{\sum_{(k_{0}, k_{1}, k_{2}) \in \mathrm{HH}, k_{0} > 0} P_{k_{0}}(\varphi^{1}_{k_{1}} \varphi^{2}_{k_{2}})}_{L^{4}_{t} \dot{H}^{-1/4}_{x}}
\aleq & \nrm{\varphi^{1}}_{S^{3/2-\gmm+\eps}} \nrm{\varphi^{2}}_{S^{\gmm-1}}
\aleq \nrm{\varphi^{1}}_{S^{\gmm}} \nrm{\varphi^{2}}_{S^{\gmm-1}},
\end{align*}
where we used $\gmm > 3/4$ in the last inequality. \qedhere
 \end{proof}

%
%

\begin{proof}[Proof of \eqref{eq:abstractEst4A:3}]
For $\abs{\nb}^{-1} P_{\leq 0} (\varphi^{1} \varphi^{2})$, the desired estimate again follows from the $L^{\infty}_{t} \dot{H}^{1/2}_{x}$ estimate in \eqref{eq:abstractEst4A:1} and H\"older in time. Hence, we may assume $k_{0} > 0$.

\pfstep{Case 1: (LH) interaction, $k_{0} > 0$, $k_{1} \leq k_{2} +5$ and $\abs{k_{0} - k_{2}} \leq 5$} 
Proceeding as in the previous proof, but this time using triangle instead of almost orthogonality (which is false for $L^{\infty}_{x}$), we have
\begin{align*}
	\nrm{\sum_{(k_{0}, k_{1}, k_{2}) \in \mathrm{LH}, k_{0} > 0} \abs{\nb}^{-1} P_{k_{0}}(\varphi^{1}_{k_{1}} \varphi^{2}_{k_{2}})}_{L^{2}_{t} L^{\infty}_{x}}
	\aleq & \sum_{k_{2} > -5} 2^{-k_{2}} \nrm{\sum_{k_{1}: k_{1} \leq k_{2}+5} \varphi^{1}_{k_{1}} \varphi^{2}_{k_{2}}}_{L^{2}_{t} L^{\infty}_{x}}
\end{align*}

We estimate each summand as follows:
\begin{align*}
2^{-k_{2}} \nrm{\sum_{k_{1}: k_{1} \leq k_{2}+5} \varphi^{1}_{k_{1}} \varphi^{2}_{k_{2}}}_{L^{2}_{t} L^{\infty}_{x}}
\aleq \nrm{\varphi^{1}}_{S^{\gmm}} 2^{-k_{2}/4}\nrm{\varphi^{2}_{k_{2}}}_{S^{0}_{k_{2}}} \sum_{k_{1}: k_{1} \leq k_{2}+5} 2^{3/4 k_{1}} 2^{-\gmm k^{+}_{1}}.
\end{align*}

As $\gmm > 3/4$, the $k_{1}$ sum is uniformly bounded by $\aleq 1$ in $k_{2}$. Summing up in $k_{2} > -5$,
\begin{align*}
\nrm{\sum_{(k_{0}, k_{1}, k_{2}) \in \mathrm{LH}, k_{0} > 0} \abs{\nb}^{-1} P_{k_{0}}(\varphi^{1}_{k_{1}} \varphi^{2}_{k_{2}})}_{L^{2}_{t} L^{\infty}_{x}}
\aleq \nrm{\varphi^{1}}_{S^{\gmm}} \sum_{k_{2} > -5} 2^{-k_{2}/4}\nrm{\varphi^{2}_{k_{2}}}_{S^{0}_{k_{2}}}  
\aleq \nrm{\varphi^{1}}_{S^{\gmm}} \nrm{\varphi^{2}}_{S^{\gmm-1}} 
\end{align*}
where we used $\gmm > 3/4$ again in the last inequality.

\pfstep{Case 2: (HL) interaction, $k_{0} > 0$, $k_{2} \leq k_{1} +5$ and $\abs{k_{0} - k_{1}} \leq 5$}
As in the previous proof, this case is more favorable than Case 1. We leave the details of this case to the reader.

\pfstep{Case 3: (HH) interaction, $0 < k_{0} \leq \min\set{k_{1}, k_{2}} - 5$ and $\abs{k_{1} - k_{2}} \leq 5$}
Let $\eps > 0$ be a small positive number, and let $2/q = 1/2 - \eps$, $4/r = \eps$. Applying triangle, H\"older in time and Sobolev, we estimate
\begin{align*}
& \nrm{\sum_{(k_{0}, k_{1}, k_{2}) \in \mathrm{HH}, k_{0} > 0} \abs{\nb}^{-1} P_{k_{0}}(\varphi^{1}_{k_{1}} \varphi^{2}_{k_{2}})}_{L^{2}_{t} L^{\infty}_{x}} \\
& \quad \aleq \sum_{k_{1}, k_{2} > 5, \abs{k_{1} - k_{2}} \leq 5} \bb( \sum_{0 < k_{0} \leq \min \set{k_{1}, k_{2}} - 5} 2^{\eps k_{0}} \nrm{ \abs{\nb}^{-1} P_{k_{0}} (\varphi^{1}_{k_{1}} \varphi^{2}_{k_{2}})}_{L^{q/2}_{t} L^{r/2}_{x}} \bb)
\end{align*}

Using Lemma \ref{lem:bilinStr} , we estimate the $k_{0}$ sum as follows:
\begin{align*}
& \sum_{k_{0}: 0 < k_{0} \leq \min \set{k_{1}, k_{2}} - 5} 2^{\eps k_{0}} \nrm{ \abs{\nb}^{-1} P_{k_{0}} (\varphi^{1}_{k_{1}} \varphi^{2}_{k_{2}})}_{L^{q/2}_{t} L^{r/2}_{x}} \\
& \quad \aleq  \sum_{k_{0}: 0 < k_{0} \leq \min \set{k_{1}, k_{2}} - 5} 2^{\eps k_{0}} 2^{k_{1}/4} \nrm{\varphi^{1}_{k_{1}}}_{S^{0}_{k_{1}}} 2^{k_{2}/4} \nrm{\varphi^{2}_{k_{2}}}_{S^{0}_{k_{2}}} \\
& \quad \aleq  2^{(3/2-\gmm+\eps)k_{1}} \nrm{\varphi^{1}_{k_{1}}}_{S^{0}_{k_{1}}} 2^{(\gmm-1)k_{2}} \nrm{\varphi^{2}_{k_{2}}}_{S^{0}_{k_{2}}}.
\end{align*}

By Cauchy-Schwarz and Lemma \ref{lem:simpleConv}, we obtain
\begin{align*}
\nrm{\sum_{(k_{0}, k_{1}, k_{2}) \in \mathrm{HH}, k_{0} > 0} \abs{\nb}^{-1} P_{k_{0}}(\varphi^{1}_{k_{1}} \varphi^{2}_{k_{2}})}_{L^{2}_{t} L^{\infty}_{x}}
\aleq & \nrm{\varphi^{1}}_{S^{3/2-\gmm+\eps}} \nrm{\varphi^{2}}_{S^{\gmm-1}}
\aleq \nrm{\varphi^{1}}_{S^{\gmm}} \nrm{\varphi^{2}}_{S^{\gmm-1}},
\end{align*}
where we used $\gmm > 3/4$ in the last inequality. \qedhere
\end{proof} 

%
%

\begin{proof}[Proof of \eqref{eq:abstractEst4A:4} and \eqref{eq:abstractEst4A:5}]
For any $\bt \in \bbR$, we have
\begin{equation*}
\nrm{(-\lap)^{-1} \big( \rd_{1} (\varphi^{1} \rd_{2}  \varphi^{2}) + \rd_{2} ( \varphi^{1} \rd_{1} \varphi^{2}) \big)}_{L^{4}_{t} \dot{H}^{\bt}_{x}}
\aleq \sum_{j=1,2} \nrm{\abs{\nb}^{-1} (\varphi^{1} \rd_{j} \varphi^{2})}_{L^{4}_{t} \dot{H}^{\bt}_{x}}.
\end{equation*}

As a consequence, \eqref{eq:abstractEst4A:5} follows from interpolating \eqref{eq:abstractEst4A:2} and \eqref{eq:abstractEst4A:4}. Thus we are only left with the task of proving \eqref{eq:abstractEst4A:4}.

Our proof will be a variant of that of \eqref{eq:abstractEst4A:2}. An important difference in the present case is that the bilinear form
\begin{equation*}
	\rd_{1} (\varphi^{1} \rd_{2} \varphi^{2}) - \rd_{2} (\varphi^{1} \rd_{1} \varphi^{2}) 
\end{equation*}
has a special cancellation structure. Indeed, from $(\rd_{1} \rd_{2} - \rd_{2} \rd_{1}) (\varphi^{1} \varphi^{2}) = 0$, we have
\begin{equation} \label{eq:nullform4A0}
	\rd_{1} (\varphi^{1} \rd_{2} \varphi^{2}) - \rd_{2} (\varphi^{1} \rd_{1} \varphi^{2}) 
	= - \rd_{1} (\rd_{2} \varphi^{1}  \varphi^{2}) + \rd_{2} (\rd_{1} \varphi^{1} \varphi^{2}).
\end{equation}

This allows us to move the derivative to always fall on the lower frequency piece.

The case $k_{0} \leq 0$ again follows from the $L^{\infty}_{t} \dot{H}^{1/2}_{x}$ estimate in \eqref{eq:abstractEst4A:1} and H\"older in time, so it suffices to consider $k_{0} > 0$. As usual, we divide into three cases according to Littlewood-Paley trichotomy. 

\pfstep{Case 1: (LH) interaction, $k_{0} > 0$, $k_{1} \leq k_{2} +5$ and $\abs{k_{0} - k_{2}} \leq 5$} 
Using \eqref{eq:nullform4A0}, we shall move the derivative to fall on $\varphi^{1}_{k_{1}}$. Then by $L^{2}$ almost orthogonality of $P_{k}$'s and Lemma \ref{lem:simpleConv}, we have
\begin{align}
& \nrm{\sum_{(k_{0}, k_{1}, k_{2}) \in \mathrm{LH}, k_{0} > 0} (-\lap)^{-1} P_{k_{0}} \big( \rd_{1} (\rd_{2} \varphi^{1}_{k_{1}}  \varphi^{2}_{k_{2}}) - \rd_{2} (\rd_{1} \varphi^{1}_{k_{1}} \varphi^{2}_{k_{2}}) \big)}_{L^{4}_{t} \dot{H}^{2(\gmm-1/2)+1/4}_{x}} \label{eq:abstractEst4A:4:pf:1} \\
& \quad \aleq \sum_{j=1,2} \bb( \sum_{k_{2} > -5} 2^{(4\gmm-7/2) k_{2}} \nrm{\sum_{k_{1} \leq k_{2}+5} \rd_{j} \varphi^{1}_{k_{1}} \varphi^{2}_{k_{2}}}_{L^{4}_{t} L^{2}_{x}}^{2} \bb)^{1/2} \notag
\end{align}

For each $j=1,2$, we estimate each summand of the above $\ell^{2}$ sum by triangle, H\"older and Strichartz as follows:
\begin{align*}
2^{(2\gmm - 7/4)k_{2}} \nrm{\sum_{k_{1}: k_{1} \leq k_{2}+5} \rd_{j} \varphi^{1}_{k_{1}} \varphi^{2}_{k_{2}}}_{L^{4}_{t} L^{2}_{x}}
\aleq & \sum_{k_{1}: k_{1} \leq k_{2}+5} 2^{k_{1}} 2^{(2\gmm - 7/4)k_{2}} \nrm{\varphi^{1}_{k_{1}}}_{L^{4}_{t} L^{\infty}_{x}} \nrm{\varphi^{2}_{k_{2}}}_{L^{\infty}_{t} L^{2}_{x}} \\
\aleq & \nrm{\varphi^{1}}_{S^{\gmm}} 2^{\gmm k_{2}} \nrm{\varphi^{2}_{k_{2}}}_{L^{\infty}_{t} L^{2}_{x}} \sum_{k_{1}: k_{1} \leq k_{2} + 5} 2^{(7/4)k_{1}} 2^{-\gmm k_{1}^{+}} 2^{(\gmm - 7/4) k_{2}}  
\end{align*}

Since $\gmm < 7/4$, the $k_{1}$ sum is bounded by $\aleq 1$ uniformly in $k_{2}$.  Summing up in $k_{2} > -5$, we may estimate \eqref{eq:abstractEst4A:4:pf:1} by $\aleq \nrm{\varphi^{1}}_{S^{\gmm}} \nrm{\varphi^{2}}_{S^{\gmm}}$ as desired.  

\pfstep{Case 2: (HL) interaction, $k_{0} > 0$, $k_{2} \leq k_{1} +5$ and $\abs{k_{0} - k_{1}} \leq 5$}
Thanks to \eqref{eq:nullform4A0}, and symmetry of the right-hand of \eqref{eq:abstractEst4A:4}, this is reduced to Case 1 by interchanging $\varphi^{1}$ and $\varphi^{2}$.

\pfstep{Case 3: (HH) interaction, $0 < k_{0} \leq \min\set{k_{1}, k_{2}} - 5$ and $\abs{k_{1} - k_{2}} \leq 5$}
This is more favorable than Case 3 in the proof of \eqref{eq:abstractEst4A:2}; in fact, an easy variant of the proof there applies here. We leave the details to the reader. \qedhere
 \end{proof}

%
%


We end this section with a simple trilinear estimate for estimating $A_{0,2}$.
\begin{lemma} \label{lem:abstractEst4A:quadric}
Let $\gmm > 3/4$ and $I \subset \bbR$ a finite interval. For $B, \varphi^{1}, \varphi^{2} \in C^{\infty}_{t}(I, \calS_{x})$, we have
\begin{equation} \label{}
\begin{aligned}
	& \nrm{\abs{\nb}^{-1} (B \varphi^{1} \varphi^{2})}_{L^{\infty}_{t} (I; \dot{H}^{1/2}_{x})}
	+\nrm{\abs{\nb}^{-1} (B \varphi^{1} \varphi^{2})}_{L^{\infty}_{t} (I; \dot{H}^{3/4}_{x})}
	 + \nrm{\abs{\nb}^{-1} (B \varphi^{1} \varphi^{2})}_{L^{\infty}_{t} (I; \dot{H}^{1}_{x} \cap L^{\infty}_{x})} \\
	& \qquad \aleq \nrm{B}_{L^{\infty}_{t}(I, \dot{H}^{1/2}_{x})} 
		\nrm{\varphi^{2}}_{S^{\gmm}(I)} \nrm{\varphi^{3}}_{S^{\gmm}(I)}. 
\end{aligned}
\end{equation}
\end{lemma}

\begin{proof} 
Fix $t \in I$; in what follows we shall suppress writing $t$, with the understanding that each function is evaluated at $t$. 
%
%
The $L^{\infty}_{t} \dot{H}^{1/2}_{x}$ estimate follows from
\begin{equation*}
\nrm{\abs{\nb}^{-1} (B \varphi^{1} \varphi^{2})}_{\dot{H}^{1/2}_{x}} \aleq \nrm{B}_{\dot{H}^{1/2}_{x}} \nrm{\varphi^{1}}_{\dot{H}^{1/2}_{x}} \nrm{\varphi^{2}}_{\dot{H}^{1/2}_{x}},
\end{equation*}
which in turn can be established using Lemma \ref{lem:SobProd}. Then by interpolation and Gagliardo-Nirenberg, the $L^{\infty} \dot{H}^{3/4}_{x}$ and $L^{\infty}_{t} (\dot{H}^{1}_{x} \cap L^{\infty}_{x})$ estimates would follow if we prove
\begin{equation*} 
	\nrm{B \varphi^{1} \varphi^{2}}_{\dot{H}^{2\eps}_{x}} \aleq \nrm{B}_{\dot{H}^{1/2}_{x}} \nrm{\varphi^{1}}_{\dot{H}^{3/4+\eps}_{x}} \nrm{\varphi^{2}}_{\dot{H}^{3/4+\eps}_{x}}
\end{equation*}
for some $\eps > 0$. Choosing $0 < \eps < \min \set{1/4, \gmm - 3/4}$, this follows again from Lemma \ref{lem:SobProd}. \qedhere
\end{proof}

\section{Estimates for $\phi$} \label{sec:est4phi}
In \S \ref{subsec:Str4wave}, we will briefly recall the wave equation for $\phi$, and apply the twisted Duhamel formula to recast it in an integral form. 
Motivated by this, in \S \ref{subsec:abstractEst4phi} we will state and prove multilinear estimates for estimating $\phi$ and $\covD_{t} \phi$.

\subsection{Structure of the wave equation for $\phi$} \label{subsec:Str4wave}
From \eqref{eq:CSH-C}, recall that the wave equation for $\phi$ may be written as
\begin{equation*}
\Box \phi + i \rd_{t} (A_{0} \phi) = m \phi + 2 i A^{\ell} \rd_{\ell} \phi - i A_{0} \covD_{t} \phi + A^{\ell} A_{\ell} \phi + W(\phi).
\end{equation*}

Applying the twisted Duhamel formula (Lemma \ref{lem:tDuhamel}), the above equation may be put in an integral form as follows:
\begin{equation*} 
\begin{aligned}
	\phi(t,x) 
	= &\cos t \abs{\nb} \phi(0, x) + \frac{\sin t \abs{\nb}}{\abs{\nb}} \covD_{t} \phi(0, x) + i \int_{0}^{t} \cos (t-t') \abs{\nb} (A_{0} \phi)(t', x) \, \ud t'\\
	& - \int_{0}^{t} \frac{\sin (t-t') \abs{\nb}}{\abs{\nb}} \bb(  m \phi + 2 i A^{\ell} \rd_{\ell} \phi - i A_{0} \covD_{t} \phi + A^{\ell} A_{\ell} \phi + V(\phi) \bb) (t', x) \, \ud t'.
\end{aligned}
\end{equation*}

Moreover, for $\covD_{t} \phi$, we have
\begin{equation*}
\begin{aligned}
	\covD_{t} \phi(t,x) 
	= &- \abs{\nb} \sin t \abs{\nb} \phi(0, x) + \cos t \abs{\nb} \covD_{t} \phi(0, x) - i \int_{0}^{t} \sin (t-t') \abs{\nb} (\abs{\nb} (A_{0} \phi))(t', x) \, \ud t'\\
	& - \int_{0}^{t} \cos (t-t') \abs{\nb} \bb( m \phi + 2 i A^{\ell} \rd_{\ell} \phi - i A_{0} \covD_{t} \phi + A^{\ell} A_{\ell} \phi + V(\phi) \bb) (t', x) \, \ud t'.
\end{aligned}
\end{equation*}

From the above formulae we see that in order to estimate $(\phi, \covD_{t} \phi)$ in $S^{\gmm} \times S^{\gmm-1}$, it suffices to put $\Box \phi + i \rd_{t} (A_{0}  \phi)$ in $L^{1}_{t} H^{\gmm-1}_{x}$ and $A_{0} \phi$ in $L^{1}_{t} H^{\gmm}_{x}$. This motivates the multilinear estimates that we prove in the following subsection.


\subsection{Multilinear estimates} \label{subsec:abstractEst4phi}
We start with an estimate that will be useful for treating the trilinear\footnote{The terms $A^{\ell} \rd_{\ell} \phi$, $A_{0} \covD_{t} \phi$ and $A_{0} \phi$ are trilinear in $\phi$, as $A_{1}, A_{2}$ and $A_{0,1}$ are obtained by solving elliptic equations with a term quadratic in $\phi$ on the right-hand side. By a similar consideration, $A^{\ell} A_{\ell} \phi$, $A_{0,2} \covD_{t} \phi$ and $A_{0,2} \phi$ are quintilinear in $\phi$.} interactions $A^{\ell} \rd_{\ell} \phi$, $A_{0} \covD_{t} \phi$ and $A_{0} \phi$, and also the quintilinear interactions $A_{0,2} \covD_{t} \phi$ and $A_{0,2} \phi$.

\begin{lemma} \label{lem:abstractEst4phi}
Let $3/4 < \gmm < 7/4$ and $I \subset \bbR$ a time interval such that $\abs{I} \leq 1$. Then for $B, B_{0}, \varphi \in C_{t}^{\infty}(I, \calS_{x})$, we have
\begin{align} 
\nrm{B \varphi}_{L^{2}_{t}(I; H^{\gmm-1}_{x})} 
\aleq & (\nrm{B}_{L^{2}_{t} (I; L^{\infty}_{x})} + \nrm{B}_{L^{4}_{t} (I; \dot{H}^{3/4}_{x})})\nrm{\varphi}_{S^{\gmm-1}(I)}, \label{eq:abstractEst4phi:1} \\
\nrm{B_{0} \varphi}_{L^{2}_{t}(I; H^{\gmm}_{x})} 
\aleq & (\nrm{B_{0}}_{L^{2}_{t} (I; L^{\infty}_{x})} + \nrm{B_{0}}_{L^{4}_{t}(I;\dot{H}^{3/4}_{x} \cap \dot{H}^{\gmm}_{x})})\nrm{\varphi}_{S^{\gmm}(I)}. \label{eq:abstractEst4phi:2}
\end{align}
\end{lemma}

\begin{proof} [Proof of \eqref{eq:abstractEst4phi:1}]
In what follows, we shall suppress writing $I$. As before, we shall decompose the output and two inputs into Littlewood-Paley pieces, and consider three cases according to Littlewood-Paley trichotomy.

\pfstep{Case 1. (LH) interaction, $k_{1} \leq k_{2}+5$ and $\abs{k_{0} - k_{2}} \leq 5$}
Using $L^{2}$ almost orthogonality of $P_{k}$'s, Lemma \ref{lem:simpleConv} and H\"older, we estimate
\begin{align*}
\nrm{\sum_{(k_{0}, k_{1}, k_{2}) \in \mathrm{LH}} P_{k_{0}} (B_{k_{1}} \varphi_{k_{2}})}_{L^{2}_{t} H^{\gmm-1}_{x}}
\aleq & \bb( \sum_{k_{2} > -5} 2^{2(\gmm-1) k^{+}_{2}}\nrm{B_{\leq k_{2}+5} \varphi_{k_{2}}}^{2}_{L^{2}_{t,x}} \bb)^{1/2} \\
\aleq & \nrm{B}_{L^{2}_{t} L^{\infty}_{x}} \bb( \sum_{k_{2} > -5} 2^{2(\gmm-1) k^{+}_{2}}\nrm{\varphi_{k_{2}}}^{2}_{L^{\infty}_{t} L^{2}_{x}} \bb)^{1/2} \\
\aleq & \nrm{B}_{L^{2}_{t} L^{\infty}_{x}} \nrm{\varphi}_{S^{\gmm-1}}.
\end{align*}

\pfstep{Case 2. (HL) interaction, $k_{2} \leq k_{1}+5$ and $\abs{k_{0} - k_{1}} \leq 5$}
Using $L^{2}$ almost orthogonality of $P_{k}$'s and Lemma \ref{lem:simpleConv}, we estimate
\begin{align*}
\nrm{\sum_{(k_{0}, k_{1}, k_{2}) \in \mathrm{HL}} P_{k_{0}} (B_{k_{1}} \varphi_{k_{2}})}_{L^{2}_{t} H^{\gmm-1}_{x}}
\aleq & \bb( \sum_{k_{1} > -5} 2^{2(\gmm-1) k^{+}_{1}}\nrm{\sum_{k_{2} \leq k_{1}+5} B_{k_{1}} \varphi_{k_{2}}}^{2}_{L^{2}_{t,x}} \bb)^{1/2}.
\end{align*}

We estimate each summand of the above $\ell^{2}$ sum by H\"older and Strichartz as follows:
\begin{align*}
2^{(\gmm-1) k^{+}_{1}}\nrm{\sum_{k_{2}:k_{2} \leq k_{1}+5} B_{k_{1}} \varphi_{k_{2}}}_{L^{2}_{t,x}}
\aleq &  2^{(3/4)k_{1}} \nrm{B_{k_{1}}}_{L^{4}_{t} L^{2}_{x}} \nrm{\varphi}_{S^{\gmm-1}}  \\
& \times	\sum_{k_{2}: k_{2} \leq k_{1} +5} 2^{-(3/4)k_{1}} 2^{(\gmm-1) k^{+}_{1}} 2^{(3/4)k_{2}} 2^{(1-\gmm) k^{+}_{2}} 
\end{align*}

As $\gmm < 7/4$, the $k_{2}$ sum is bounded by $\aleq 1$ uniformly in $k_{1}$. Evaluating the $\ell^{2}$ sum in $k_{1}$, we conclude
\begin{equation*}
\nrm{\sum_{(k_{0}, k_{1}, k_{2}) \in \mathrm{HL}} P_{k_{0}} (B_{k_{1}} \varphi_{k_{2}})}_{L^{2}_{t} H^{\gmm-1}_{x}}
\aleq \nrm{B}_{L^{4}_{t} \dot{H}^{3/4}_{x}} \nrm{\varphi}_{S^{\gmm-1}}.
\end{equation*} 

\pfstep{Case 3. (HH) interaction, $k_{0} \leq \min \set{k_{1}, k_{2}} -5$ and $\abs{k_{1} - k_{2}} \leq 5$}
This case follows immediately from Lemma \ref{lem:bilinVarStr} (with $\sgm = \gmm - 1$). \qedhere
\end{proof}

\begin{proof} [Proof of \eqref{eq:abstractEst4phi:2}]
For convenience, we shall suppress writing $I$, and also abbreviate $B_{0}$ to $B$. We shall follow the preceding proof very closely; indeed, let us again decompose the output and inputs into Littlewood-Paley pieces, and analyze them according to Littlewood-Paley trichotomy.

\pfstep{Case 1. (LH) interaction, $k_{1} \leq k_{2}+5$ and $\abs{k_{0} - k_{2}} \leq 5$}
Proceeding as in Case 1 of the proof of \eqref{eq:abstractEst4phi:1}, we easily obtain the estimate
\begin{equation*}
	\nrm{\sum_{(k_{0}, k_{1}, k_{2}) \in \mathrm{LH}} P_{k_{0}} (B_{k_{1}} \varphi_{k_{2}})}_{L^{2}_{t} H^{\gmm}_{x}}
	\aleq \nrm{B}_{L^{2}_{t} L^{\infty}_{x}} \nrm{\varphi}_{L^{\infty}_{t} H^{\gmm}_{x}}
	\aleq \nrm{B}_{L^{2}_{t} L^{\infty}_{x}} \nrm{\varphi}_{S^{\gmm}}.
\end{equation*}

\pfstep{Case 2. (HL) interaction, $k_{2} \leq k_{1}+5$ and $\abs{k_{0} - k_{1}} \leq 5$}
By $L^{2}$ almost orthogonality of $P_{k}$'s and Lemma \ref{lem:simpleConv}, we have
\begin{align*}
\nrm{\sum_{(k_{0}, k_{1}, k_{2}) \in \mathrm{HL}} P_{k_{0}} (B_{k_{1}} \varphi_{k_{2}})}_{L^{2}_{t} H^{\gmm}_{x}}
\aleq & \bb( \sum_{k_{1} > -5} 2^{2 \gmm k^{+}_{1}}\nrm{\sum_{k_{2} \leq k_{1}+5} B_{k_{1}} \varphi_{k_{2}}}^{2}_{L^{2}_{t,x}} \bb)^{1/2}.
\end{align*}

We estimate each summand of the above $\ell^{2}$ sum by H\"older and Strichartz as follows:
\begin{align*}
2^{\gmm k^{+}_{1}}\nrm{\sum_{k_{2}:k_{2} \leq k_{1}+5} B_{k_{1}} \varphi_{k_{2}}}_{L^{2}_{t,x}}
\aleq &  2^{(3/4)k_{1}} 2^{(\gmm-3/4) k^{+}_{1}} \nrm{B_{k_{1}}}_{L^{4}_{t} L^{2}_{x}} \nrm{\varphi}_{S^{\gmm}}  \\
& \times	\sum_{k_{2}: k_{2} \leq k_{1} +5} 2^{-(3/4)k_{1}} 2^{(3/4) k^{+}_{1}} 2^{(3/4)k_{2}} 2^{-\gmm k^{+}_{2}} 
\end{align*}

Since $\gmm > 3/4$, the $k_{2}$ sum is bounded by $\aleq 1$ uniformly in $k_{1}$. Evaluating the $\ell^{2}$ sum in $k_{1}$, we obtain
\begin{equation*}
\nrm{\sum_{(k_{0}, k_{1}, k_{2}) \in \mathrm{HL}} P_{k_{0}} (B_{k_{1}} \varphi_{k_{2}})}_{L^{2}_{t} H^{\gmm}_{x}}
\aleq \nrm{B}_{L^{4}_{t} ( \dot{H}^{3/4}_{x} \cap \dot{H}^{\gmm}_{x})} \nrm{\varphi}_{S^{\gmm}}.
\end{equation*} 

\pfstep{Case 3. (HH) interaction, $k_{0} \leq \min \set{k_{1}, k_{2}} -5$ and $\abs{k_{1} - k_{2}} \leq 5$}
As before, this case follows immediately from Lemma \ref{lem:bilinVarStr} (with $\sgm = \gmm$). \qedhere
\end{proof}

The next estimate will be used to handle the quintilinear interaction $A^{\ell} A_{\ell} \phi$.
\begin{lemma} \label{lem:abstractEst4phi:quintic}
Let $3/4 < \gmm < 1$ and $I \subset \bbR$ a finite interval. Then for $B^{1}, B^{2}, \varphi \in C^{\infty}_{t}(I; \calS_{x})$, we have
\begin{equation} \label{eq:abstractEst4phi:3}
	\nrm{B^{1} B^{2} \varphi}_{L^{\infty}_{t} (I; H^{\gmm-1}_{x})} 
	\aleq \nrm{B^{1}}_{L^{\infty}_{t} (I; \dot{H}^{1/2}_{x})} \nrm{B^{2}}_{L^{\infty}_{t} (I; \dot{H}^{1/2}_{x})} \nrm{\varphi}_{S^{\gmm}(I)}.
\end{equation}

For $\gmm = 1$, we have the following substitute for every $0< \bt < 1$:
\begin{equation} \label{eq:abstractEst4phi:4}
	\nrm{B^{1} B^{2} \varphi}_{L^{\infty}_{t} (I; L^{2}_{x})} 
	\aleq \nrm{B^{1}}_{L^{\infty}_{t} (I; \dot{H}^{1-\bt/2}_{x})} \nrm{B^{2}}_{L^{\infty}_{t} (I; \dot{H}^{1-\bt/2}_{x})} \nrm{\varphi}_{L^{\infty}_{t}(I; \dot{H}^{\bt}_{x})}.
\end{equation}
\end{lemma}
\begin{proof} 
Both estimates are easy consequences of Lemma \ref{lem:SobProd}; we leave the details to the reader. \qedhere
%
\end{proof}

Finally, we state a lemma for estimating the self-interaction potential.
\begin{lemma} \label{lem:abstractEst4phi:Vphi}
Let $3/4 < \gmm \leq 1$, $I \subset \bbR$ a finite interval and $N$ an integer such that $1 \leq N < 1 +\frac{2}{1-\gmm}$ when $\gmm <1$, and simply $N \geq 1$ when $\gmm = 1$. Then for $\varphi^{j} \in C^{\infty}_{t} (I; \calS_{x})$ $(1 \leq j \leq N)$, 
\begin{equation*}
	\nrm{\prod_{j=1}^{N} \varphi^{j}}_{L^{1}_{t} (I; H^{\gmm-1}_{x})} \aleq \abs{I}^{\alp} \prod_{j=1}^{N} \nrm{\varphi^{j}}_{S^{\gmm}(I)}
\end{equation*}
for some $\alp = \alp(N, \gmm) > 0$.
\end{lemma}

\begin{proof} 
In what follows, we shall suppress writing $I$. When $N = 1$ the lemma is obvious with $\alp =1$, so we may assume that $N > 1$. The case $\gmm = 1$ is also easy to discard with $\alp =1$; indeed, simply apply H\"older to put each $\varphi^{j}$ in $L^{\infty}_{t} L^{2N}_{x}$, as $S^{1} \subset L^{\infty}_{t} L^{2N}_{x}$ by Sobolev.

Consider now the case $3/4 < \gmm < 1$. Using Lemma \ref{lem:SobProd}, we first estimate
\begin{equation*}
\nrm{\prod_{j=1}^{N} \varphi^{j}}_{L^{1}_{t} H^{\gmm-1}_{x}} \aleq \nrm{\prod_{j=1}^{N-1} \varphi^{j}}_{L^{1}_{t} L^{2}_{x}} \nrm{\varphi^{N}}_{H^{\gmm}_{x}}
\end{equation*}

We then estimate $\nrm{\prod_{j=1}^{N-1} \varphi^{j}}_{L^{1}_{t} L^{2}_{x}}$ via the following procedure:
\begin{itemize}
\item When $N - 1 \leq \frac{1}{1-\gmm}$, we may put each $\varphi^{j}$ in $L^{\infty}_{t} L^{2(N-1)}_{x}$, as $S^{\gmm} \subset L^{\infty}_{t} L^{2(N-1)}_{x}$ by Sobolev. 
\item When $N - 1> \frac{1}{1-\gmm}$ and $\frac{1}{1-\gmm} \in \bbZ$, put the first $\frac{1}{1-\gmm}$ factors in $L^{\infty}_{t} L^{\frac{2}{1-\gmm}}_{x}$ and the remaining factors in $L^{\frac{1}{1-\gmm+\eps}}_{t} L^{\infty}_{x}$ for some small $\eps > 0$ to be chosen. Note that $S^{\gmm} \subset L^{\infty}_{t} L^{\frac{2}{1-\gmm}}_{x}$ by Sobolev and $S^{\gmm} \subset L^{\frac{1}{1-\gmm+\eps}}_{t} L^{\infty}_{x}$ by Strichartz.
\item Finally, when $N - 1> \frac{1}{1-\gmm}$ and $\frac{1}{1-\gmm} \not\in \bbZ$, we again put the first $\lfloor \frac{1}{1-\gmm} \rfloor$ factors in $L^{\infty}_{t} L^{\frac{2}{1-\gmm}}_{x}$, but put the next factor in $L^{q}_{t} L^{r}_{x}$, where $\frac{1}{r} = \frac{1}{2} - \frac{1-\gmm}{2} \lfloor \frac{1}{1-\gmm} \rfloor$ and $\frac{1}{q} = 1 - \frac{2}{r} - \gmm$. We then put the remaining factors in $L^{\frac{1}{1-\gmm+\eps}}_{t} L^{\infty}_{x}$ as before, where $\eps > 0$ is again a small number to be chosen. Note that $S^{\gmm} \subset L^{q}_{t} L^{r}_{x}$ by Strichartz since $\gmm > 3/4$.
%
\end{itemize}

In the former case, $\alp =1$. In the second and third cases, the hypothesis $N < 1+ \frac{2}{1-\gmm}$ ensures that (after choosing $\eps > 0$ sufficiently small) we are left with a factor of $\abs{I}^{\alp}$ with $\alp > 0$. \qedhere
%
%
\end{proof}

\section{Proof of the main theorems} \label{sec:pf4mainThm}
In this section, we finally prove Theorems \ref{thm:lwp4CSH} and \ref{thm:gwp4CSH}. The first step is to formulate a suitable version of energy estimate using the twisted Duhamel formula, using the machinery we have developed so far. Given $\gmm \geq 0$ and a finite interval $I \subset \bbR$, define the space $\mathfrak{A}^{\gmm}_{0}(I)$ by
\begin{equation*}
	\mathfrak{A}^{\gmm}_{0}(I) := L^{4}_{t} (I; \dot{H}^{3/4}_{x}) \cap L^{4}_{t} (I; \dot{H}^{\gmm}_{x}) \cap L^{2}_{t} (I; C^{0}_{x}).
\end{equation*}
where $C^{0}_{x}$ is the closure of $\calS_{x}$ under the sup-norm.

\begin{proposition}[Energy estimate via twisted Duhamel formula] \label{prop:tDuhamelEnergy}
Let $3/4 < \gmm \leq 1$ and $I \subset \bbR$ a finite interval such that $0 \in I$. Then the following statements hold.
\begin{enumerate}
\item Consider the IVP
\begin{equation} \label{eq:rnWave:IVP}
\left\{	
\begin{aligned}
	\Box \phi + i \rd_{t} (A_{0} \phi) =& F \\
	(\phi, \covD_{t} \phi)(0) =& (f, g),
\end{aligned}
\right.
\end{equation}
where $(f, g) \in H^{\gmm}_{x} \times H^{\gmm-1}_{x}$, $F \in L^{1}_{t} (I; H^{\gmm-1}_{x})$ and $A_{0} \in C_{t}(I; \dot{H}^{1/2}_{x}) \cap \mathfrak{A}^{\gmm}_{0}(I)$.
Then there exists $\dlt_{0} > 0$ such that the following statement holds: If
\begin{equation*}
	\abs{I} \leq \dlt_{0} \nrm{A_{0}}^{-2}_{\mathfrak{A}^{\gmm}_{0}(I)},
\end{equation*}
then there exists a unique solution $\phi$ to the above system such that $(\phi, \covD_{t} \phi) \in S^{\gmm}(I) \times S^{\gmm-1}(I)$, which obeys the estimate
\begin{equation} \label{eq:tDuhamelEnergy}
	\nrm{\phi}_{S^{\gmm}(I)} + \nrm{\covD_{t} \phi}_{S^{\gmm-1}(I)} \aleq \nrm{(f,g)}_{H^{\gmm}_{x} \times H^{\gmm-1}_{x}} + \nrm{F}_{L^{1}_{t}(I; H^{\gmm-1}_{x})}.
\end{equation}

\item Consider $(f', g') \in H^{\gmm}_{x} \times H^{\gmm-1}_{x}$, $F' \in L^{1}_{t} (I; H^{\gmm-1}_{x})$ and $A'_{0} \in C_{t}(I; \dot{H}^{1/2}_{x}) \cap \mathfrak{A}^{\gmm}_{0}(I)$ satisfying $\abs{I} \leq \dlt_{0} \nrm{A'_{0}}^{-2}_{\mathfrak{A}^{\gmm}_{0}(I)}$. Let $\phi'$ be the corresponding solution to \eqref{eq:rnWave:IVP} given by (1). Then the following difference estimate holds:
\begin{equation} \label{eq:tDuhamelEnergy:diff}
\begin{aligned}
&		\nrm{\dlt \phi}_{S^{\gmm}(I)} + \nrm{\dlt \covD_{t} \phi}_{S^{\gmm-1}(I)} \\
& \quad 	\aleq \nrm{\dlt(f,g)}_{H^{\gmm}_{x} \times H^{\gmm-1}_{x}}  
		+ \nrm{\dlt F}_{L^{1}_{t}(I; H^{\gmm-1}_{x})} + \abs{I}^{1/2} \nrm{\dlt A_{0}}_{\mathfrak{A}^{\gmm}_{0}(I)} \nrm{\phi'}_{S^{\gmm}(I)} 
\end{aligned}
\end{equation}

Here, $\dlt$ is a shorthand for the difference between primed and un-primed objects. e.g. $\dlt \phi := \phi - \phi'$, $\dlt \covD_{t} \phi := \covD_{t} \phi - \covD_{t}' \phi'$, $\dlt (f,g) := (f,g) - (f',g')$ etc.
\end{enumerate}
\end{proposition}

\begin{proof} 
We begin by considering $f,g, f', g' \in \calS_{x}$ and $A_{0}, F, A'_{0}, F' \in C^{\infty}_{t}(I; \calS_{x})$. Well-posedness of the standard (inhomogeneous) wave equation, along with a simple Gronwall inequality, easily shows the existence of a solution $\phi \in C_{t}^{\infty}(I; \calS_{x})$ of \eqref{eq:rnWave:IVP}. Then by the twisted Duhamel formula (Lemma \ref{lem:tDuhamel}), Lemmas \ref{lem:homEnergy:1}, \ref{lem:homEnergy:2} and Corollary \ref{cor:inhomEnergy}, there exists some $C_{1} > 0$ such that
\begin{equation*}
	\nrm{\phi}_{S^{\gmm}(I)} + \nrm{\covD_{t} \phi}_{S^{\gmm-1}(I)} \leq C_{1} \bb( \nrm{(f,g)}_{H^{\gmm}_{x} \times H^{\gmm-1}_{x}} + \nrm{F}_{L^{1}_{t}(I; H^{\gmm-1}_{x})} + \abs{I}^{1/2}  \nrm{A_{0} \phi}_{L^{2}_{t} (I; H^{\gmm}_{x})} \bb).
\end{equation*}

By \eqref{eq:abstractEst4phi:2} (Lemma \ref{lem:abstractEst4phi}), there exists $C_{2} > 0$ such that
\begin{equation} \label{eq:tDuhamelEnergy:pf}
C_{1} \abs{I}^{1/2}  \nrm{A_{0} \phi}_{L^{2}_{t}(I; H^{\gmm}_{x})} \leq C_{1} C_{2} \abs{I}^{1/2} \nrm{A_{0}}_{\mathfrak{A}^{\gmm}_{0}(I)} \nrm{\phi}_{S^{\gmm}(I)}.
\end{equation}

Choosing $\dlt_{0}$ sufficiently small so that $C_{1} C_{2} \abs{I}^{1/2} \nrm{A_{0}}_{\mathfrak{A}^{\gmm}_{0}(I)} \leq 1/2$, this term maybe absorbed into the left-hand side, proving \eqref{eq:tDuhamelEnergy}. The difference estimate \eqref{eq:tDuhamelEnergy:diff} can be proved in a similar manner by taking the difference between the twisted Duhamel formulas for $\phi$ and $\phi'$; we leave the easy details to the reader.

To remove the smoothness assumption, note that all of the statements of the proposition follow immediately by approximation, except uniqueness. To prove uniqueness, it suffices to establish the following claim: If $A_{0} \in \mathfrak{A}^{\gmm}_{0}(I)$, $(\phi, \covD_{t} \phi) \in S^{\gmm}(I) \times S^{\gmm-1}(I)$, $(\phi, \covD_{t} \phi)(0) = 0$ and $\Box \phi + i \rd_{t}(A_{0} \phi) = 0$, then $\phi = 0$. For this purpose, note that $A_{0} \phi \in L^{2}_{t} H^{\gmm}_{x}$ by \eqref{eq:tDuhamelEnergy:pf}. This is enough to justify the integration by parts in the proof of \eqref{eq:tDuhamel:1}, and thus
\begin{equation*}
	\nrm{\phi}_{S^{\gmm}(I)} 
	\leq C_{1} C_{2} \abs{I}^{1/2} \nrm{A_{0}}_{\mathfrak{A}_{0}^{\gmm}} \nrm{\phi}_{S^{\gmm}(I)} 
	\leq \frac{1}{2} \nrm{\phi}_{S^{\gmm}(I)},
\end{equation*}
from which our claim follows. \qedhere
\end{proof}

We are now ready to prove Theorems \ref{thm:lwp4CSH} and \ref{thm:gwp4CSH}. For the LWP theorem, the argument is a Picard iteration with Proposition \ref{prop:tDuhamelEnergy}. As this is a standard procedure, we will only give a sketch of the proof.

\begin{proof} [Proof of Theorem \ref{thm:lwp4CSH}]
Fix $3/4 < \gmm \leq 1$.
Starting from $(A_{0}^{(-1)}, A_{1}^{(-1)}, A_{2}^{(-1)}, \phi^{(-1)}, \covD_{t} \phi^{-1}) = 0$, we define the $n$-th Picard iterate recursively as follows: Define $A^{(n)}_{1}$ and $A^{(n)}_{2}$ by
\begin{align*}
A^{(n)}_{1} :=& - (-\lap)^{-1} \rd_{2} \Im (\phi^{(n-1)} \overline{\covD_{t} \phi^{(n-1)}}), \\
A^{(n)}_{2} :=& (-\lap)^{-1} \rd_{1} \Im (\phi^{(n-1)} \overline{\covD_{t} \phi^{(n-1)}}).
\end{align*}

Then we define $A^{(n)}_{0} := A^{(n)}_{0,1} + A^{(n)}_{0,2}$, where
\begin{align*}
A^{(n)}_{0,1} :=& (-\lap)^{-1} \rd_{1} \Im(\phi^{(n-1)} \overline{\rd_{2} \phi^{(n-1)}}) - (-\lap)^{-1}  \rd_{2} \Im(\phi^{(n-1)} \overline{\rd_{1} \phi^{(n-1)}}), \\
A^{(n)}_{0,2} :=& - (-\lap)^{-1} \rd_{1} \Re(\phi^{(n-1)} \overline{A^{(n)}_{2} \phi^{(n-1)}}) + (-\lap)^{-1}  \rd_{2} \Re(\phi^{(n-1)} \overline{A^{(n)}_{1}\phi^{(n-1)}}).
\end{align*}

We note that the operators $(-\lap)^{-1} \rd_{j}$ are well-defined if $(\phi^{(n-1)}, \covD_{t} \phi^{(n-1)}) \in S^{\gmm}(I) \times S^{\gmm-1}(I)$ thanks to Lemmas \ref{lem:abstractEst4A:1} and \ref{lem:abstractEst4A:quadric}. Furthermore, by Lemmas \ref{lem:abstractEst4A:2} and \ref{lem:nullStr4A0}, we see that $A_{0} \in \mathfrak{A}^{\gmm}_{0}(I)$ as well.

Next, let $\phi^{(n)}$ be the solution to the following equation with initial data $(f,g)$ obtained by Proposition \ref{prop:tDuhamelEnergy}:
\begin{equation} \label{eq:lwp4CSH:pf:1}
\begin{aligned}
\Box \phi^{(n)} + i \rd_{t} (A^{(n)}_{0} \phi^{(n)}) 
=& m \phi^{(n-1)} + 2 i A^{(n)}_{\ell} \rd^{\ell} \phi^{(n-1)} - i A^{(n)}_{0,1} \covD_{t} \phi^{(n-1)}  \\
& - i A^{(n)}_{0,2} \covD_{t} \phi^{(n-1)} + (A^{(n)})^{\ell} A^{(n)}_{\ell} \phi^{(n-1)} + W(\phi^{(n-1)})
\end{aligned}
\end{equation}

This belongs to $S^{\gmm}(I)$ if $(\phi^{(n-1)}, \covD_{t} \phi^{(n-1)}) \in S^{\gmm}(I) \times S^{\gmm-1}(I)$. Finally, let $\covD_{t} \phi^{(n)} := \rd_{t} \phi^{(n)} + i A_{0}^{(n)} \phi^{(n)}$, which belongs to $S^{\gmm-1}(I)$ thanks to Proposition \ref{prop:tDuhamelEnergy}.

Write $\calI := \nrm{(f, g)}_{H^{\gmm}_{x} \times H^{\gmm-1}_{x}}$. We claim that the above procedure defines\footnote{Observe that only $\phi^{(n-1)}$ and $A^{(n-1)}_{0}$ are needed to construct the whole $n$-th iterate ($A^{(n)}_{1}$, $A^{(n)}_{2}$,$A^{(n)}_{0}$, $\phi^{(n)}$, $\covD_{t} \phi^{(n)}$).} a contraction map in the set
\begin{equation*}
	\mathfrak{B}_{R, \dlt} := \set{(A_{0}, \phi) \in \mathfrak{A}^{\gmm}_{0} \times S^{\gmm}(I): 
		\nrm{(\phi, \covD_{t} \phi)}_{S^{\gmm}(I) \times S^{\gmm-1}(I)} \leq R, 
		\nrm{A_{0}}_{\mathfrak{A}^{\gmm}_{0}} \leq B (R^{2}+R^{4}),
		\abs{I} \leq \dlt }
\end{equation*}
for an absolute constant $B > 0$, $R = R(\calI) > 0$ and $\dlt = \dlt(m, R, \deg V) > 0$. It is a standard procedure to obtain Theorem \ref{thm:lwp4CSH} from this claim, whose details we omit. 
Instead of giving a full proof of the claim, we will only outline the proof that this iteration maps $\mathfrak{B}_{R, \dlt}$ to itself for some well-chosen $R, \dlt  > 0$. The proof that its Lipschitz norm is $< 1$ (by modifying the choice of $\dlt$ if necessary) is similar and therefore left to the reader.

Let $(A^{(n-1)}_{0}, \phi^{(n-1)}) \in \mathfrak{B}_{R, \dlt}$, and assume that $\abs{I} \leq \dlt \leq 1$. From the hierarchical definitions of $A^{(n)}_{1}$, $A^{(n)}_{2}$ and $A^{(n)}_{0}$, we see that Lemmas \ref{lem:abstractEst4A:1}, \ref{lem:abstractEst4A:2}, \ref{lem:nullStr4A0} and \ref{lem:abstractEst4A:quadric} (combined with H\"older in time) imply the existence of some $ B>0$ such that
\begin{equation}
	\nrm{A^{(n)}_{0}}_{\mathfrak{A}^{\gmm}_{0}} \leq B (R^{2} + R^{4}).
\end{equation}

To estimate $(\phi^{(n)}, \covD_{t} \phi^{(n)})$, note that the multilinear estimates we have proven so far (combined with H\"older in time) imply the following estimates:
\begin{align*}
\nrm{m \phi^{(n-1)}}_{L^{1}_{t} (I; H^{\gmm-1}_{x})} \aleq & m \dlt R, \\
\nrm{2 i A^{(n)}_{\ell} \rd^{\ell} \phi^{(n-1)} - i A_{0,1}^{(n)} \covD_{t} \phi^{(n-1)}}_{L^{1}_{t} (I; H^{\gmm-1}_{x})} \aleq & \dlt^{1/2} R^{3}, \hskip1.3em 
		(\hbox{Lemmas \ref{lem:abstractEst4A:2}, \ref{lem:abstractEst4phi}}) \\
\nrm{- i A_{0,2}^{(n)} \covD_{t} \phi^{(n-1)} + (A^{(n)})^{\ell} A^{(n)}_{\ell} \phi^{(n-1)}}_{L^{1}_{t} (I; H^{\gmm-1}_{x})} \aleq & \dlt R^{5}, \hskip2.5em 
		(\hbox{Lemmas \ref{lem:abstractEst4A:1}, \ref{lem:abstractEst4A:quadric}, \ref{lem:abstractEst4phi} and \ref{lem:abstractEst4phi:quintic} }) \\
\nrm{W(\phi^{(n-1)})}_{L^{1}_{t} (I; H^{\gmm-1}_{x})} \aleq & \dlt^{\alp(\deg V, \gmm)} (R+R^{2 \deg V -1}). \quad 
		(\hbox{Lemma \ref{lem:abstractEst4phi:Vphi}})
\end{align*}

These are exactly the terms on the right-hand side of \eqref{eq:lwp4CSH:pf:1}. Now applying Proposition \ref{prop:tDuhamelEnergy} with $\dlt \leq \dlt_{0} B^{-2} (R^{2} + R^{4})^{-2}$, we obtain
\begin{equation*}
	\nrm{(\phi^{(n)}, \covD_{t} \phi^{(n)})}_{S^{\gmm}(I) \times S^{\gmm-1}(I)} 
	\aleq \calI + m \dlt R + \dlt^{1/2} R^{3} + \dlt R^{5} + \dlt^{\alp} (R+R^{2 \deg V -1}).
\end{equation*}

Choosing $R = C \calI$ for sufficiently large $C > 0$ and $\dlt = \dlt(R, m, \deg V) > 0$ sufficiently small, we obtain $(A^{(n)}_{0}, \phi^{(n)}) \in \mathfrak{B}_{R, \dlt}$ as desired. \qedhere
\end{proof}

To prove finite energy GWP, we need the following lemma.
\begin{lemma} \label{lem:energy}
Let $(A_{\mu}, \phi)$ be a solution to \eqref{eq:CSH-C} on some finite interval $I \subset \bbR$ such that energy conservation holds, i.e., \eqref{eq:energyConsv} is true. Assume, moreover, that for some $\alp \geq 0$, 
\begin{equation*}
	V(r) \geq - \alp^{2} r
\end{equation*}
for all $r \geq 0$. Then the $H^{1}_{x} \times L^{2}_{x}$ norm of $(\phi, \covD_{t} \phi)$ is uniformly bounded on $I$.
\end{lemma}
\begin{proof}
We will follow the ideas of \cite[Proof of Theorem 2.3]{Selberg:2012vb}. From the hypothesis, it follows that
\begin{equation} \label{eq:energy:pf:1}
	\nrm{\covD_{t} \phi(t)}^{2}_{L^{2}_{x}} + \sum_{j = 1,2} \nrm{\covD_{j}\phi(t)}^{2}_{L^{2}_{x}} + m \nrm{\phi(t)}^{2}_{L^{2}_{x}}
	\leq 2 \bfE(t) + \alp^{2} \nrm{\phi(t)}^{2}_{L^{2}_{x}}.
\end{equation}

Using this, we shall first show that $\nrm{\phi(t)}_{L^{2}}$ is uniformly bounded on $I$. When $m > 0$ this follows immediately from \eqref{eq:energy:pf:1}; thus it suffices to consider $m = 0$. Integrating the identity $\rd_{t} \abs{\phi(t)}^{2} = 2 \Re (\phi \overline{\covD_{t} \phi})$ in time and applying Cauchy-Schwarz and \eqref{eq:energy:pf:1}, we have
\begin{align*}
	\nrm{\phi(t)}^{2} - \nrm{\phi(0)}^{2} 
	\leq & 2 \int_{0}^{t} \nrm{\phi(t')}_{L^{2}_{x}} \nrm{\covD_{t} \phi(t')}_{L^{2}_{x}} \, \ud t' \\
	\leq & \int_{0}^{t} (1+\alp^{2})\nrm{\phi(t')}^{2}_{L^{2}_{x}} \, \ud t' + 2 t \bfE.
\end{align*}

Applying Gronwall, we conclude $\sup_{t \in I} \nrm{\phi(t)}_{L^{2}_{x}} < \infty$. Then combined with \eqref{eq:energy:pf:1}, we see that
\begin{equation} \label{eq:energy:pf:2}
\sup_{t \in I} ( \nrm{\covD_{t} \phi(t)}^{2}_{L^{2}_{x}} + \sum_{j = 1,2} \nrm{\covD_{j}\phi(t)}^{2}_{L^{2}_{x}} + \nrm{\phi}^{2}_{L^{2}_{x}} ) < \infty
\end{equation}

We are now only left to show that $\nrm{\phi(t)}_{H^{1}_{x}}$ is uniformly bounded on $I$. In view of the identity $\covD_{j} \phi = \rd_{j} \phi - i A_{j} \phi$, it suffices to prove $\sup_{t \in I} \nrm{A_{j} \phi(t)}_{L^{2}_{x}} < \infty$ for $j=1,2$. 

Observe that $\sup_{t \in I} \nrm{\phi(t)}_{L^{4}_{x}} < \infty$, which follows from \eqref{eq:energy:pf:2} by the Kato inequality $\rd_{j} \abs{\phi } \leq \abs{\covD_{j} \phi}$ and Gagliardo-Nirenberg interpolation. Then from the elliptic equation for $A_{j}$ in \eqref{eq:CSH-C}, for each $t \in I$ we have 
\begin{equation*}
	\nrm{A_{j}(t)}_{L^{4}_{x}} \aleq \nrm{\Im(\phi \covD_{t} \phi)(t)}_{L^{4/3}_{x}} \leq \nrm{\phi(t)}_{L^{4}_{x}} \nrm{\covD_{t} \phi(t)}_{L^{2}_{x}}
\end{equation*}
where the right-hand side is uniformly bounded in $t \in I$. By H\"older, it then follows that $\sup_{t \in I} \nrm{A_{j} \phi(t)}_{L^{2}_{x}} < \infty$ for $j=1,2$ as desired. \qedhere
\end{proof}

From Theorem \ref{thm:lwp4CSH} and Lemma \ref{lem:energy}, finite energy GWP (Theorem \ref{thm:gwp4CSH}) now follows as a simple corollary.

\section*{Acknowledgement}
The author is indebted to H. Huh for many helpful discussions on Chern-Simons theories. This work had been initiated while the author was a graduate student at Princeton University, during which he was supported by the Samsung Scholarship. The author is a Miller Research Fellow, and graciously acknowledges the Miller Institute at UC Berkeley for the support. 
\bibliographystyle{amsplain}
\providecommand{\bysame}{\leavevmode\hbox to3em{\hrulefill}\thinspace}
\providecommand{\MR}{\relax\ifhmode\unskip\space\fi MR }
\providecommand{\MRhref}[2]{%
  \href{http://www.ams.org/mathscinet-getitem?mr=#1}{#2}
}
\providecommand{\href}[2]{#2}


\end{document}